\documentclass[a4paper,reqno,11pt]{amsart}
\usepackage[margin=1.2in]{geometry}

\usepackage{amssymb, enumerate}
\usepackage[all]{xy}
\usepackage{enumitem}
\usepackage{mathrsfs}
\usepackage{amscd}
\usepackage[%
 setpagesize=false,%
 bookmarks=true,%
 bookmarksdepth=tocdepth,%
 bookmarksnumbered=true,%
 colorlinks=false,%
 pdftitle={},%
 pdfsubject={},%
 pdfauthor={},%
 pdfkeywords={}%
]{hyperref}

\usepackage{xcolor}
\hypersetup{
	colorlinks=true,
	citecolor=blue,
	linkcolor=red,
	urlbordercolor=cyan,
}

\theoremstyle{plain}
\newtheorem{thm}{Theorem}[section] 
\newtheorem{cor}[thm]{Corollary}
\newtheorem{prop}[thm]{Proposition}
\newtheorem{conj}[thm]{Conjecture}
\newtheorem{lem}[thm]{Lemma}
\newtheorem{mainthm}[thm]{Main Theorem}

\theoremstyle{definition} 
\newtheorem{defn}[thm]{Definition}

\newtheorem{question}[thm]{Question}

\newtheorem{rem}[thm]{Remark}

\theoremstyle{remark}

\newtheorem*{acknowledgement}{Acknowledgments}

\def\phi{\varphi}
\def\epsilon{\varepsilon}

\def\to{\longrightarrow}

\newcommand{\sO}{\mathcal{O}}

\newcommand{\F}{\mathbb{F}}
\newcommand{\N}{\mathbb{N}}
\newcommand{\Q}{\mathbb{Q}} 
 
\newcommand{\R}{\mathbb{R}} 
\newcommand{\Z}{\mathbb{Z}}
\newcommand{\PP}{\mathbb{P}}

\newcommand{\bR}{\mathbf{R}} 

\newcommand{\J}{\mathcal{J}}

\newcommand{\ideala}{\mathfrak{a}}

\newcommand{\m}{\mathfrak{m}}

\newcommand{\q}{\mathfrak{q}}

\newcommand{\Bs}{\mathrm{Bs}}
\newcommand{\B}{\mathbf{B}}
\newcommand{\Bp}{\mathbf{B}_{+}}
\newcommand{\Bm}{\mathbf{B}_{-}}
\newcommand{\NNef}{\mathrm{NNef}}
\newcommand{\num}{\mathrm{num}}
\newcommand{\ord}{\mathrm{ord}}
\newcommand{\nord}{\overline{\mathrm{ord}}}
\newcommand{\BCM}{\mathrm{BCM}}

\newcommand{\Tr}{\mathrm{Tr}}

\newcommand{\taubcm}{\tau_{\mathscr{B}}}

\newcommand{\bcmt}{\mathrm{bcmt}}

\newcommand{\gfst}{\mathrm{gfst}}
\newcommand{\alphabcm}{\alpha_{\BCM}}

\newcommand{\Supp}{\mathrm{Supp}}
\newcommand{\Coker}{\mathrm{Coker}}
\newcommand{\Image}{\mathrm{Image}}
\newcommand{\Pic}{\mathrm{Pic}}
\newcommand{\divi}{\mathrm{div}}
\newcommand{\Zero}{\mathcal{Z}}

\newcommand{\Bigcone}{\mathrm{Big}}

\newcommand{\Rhat}{\widehat{R}}
\newcommand{\Rphat}{\widehat{R^+}}
\newcommand{\Deltahat}{\widehat{\Delta}}
\newcommand{\gp}{\mathrm{gp}}

\newsavebox{\circlebox}
\savebox{\circlebox}{\fontencoding{OMS}\selectfont\Large\char13}
\newlength{\circleboxwdht}

\def\Hom{\operatorname{Hom}}
\def\Spec{\operatorname{Spec}}

\def\Supp{\operatorname{Supp}}

\def\Pic{\operatorname{Pic}}

\title{NON-NEF LOCI IN MIXED CHARACTERISTIC}

\author{Ryotaro Iwane}
\address{Graduate School of Mathematical Sciences, University of Tokyo, 3-8-1 Komaba, Meguro-ku, Tokyo 153-8914, Japan}
\email{1733393470@g.ecc.u-tokyo.ac.jp}

\thanks{}

\subjclass[2020]{}

\begin{document}

\begin{abstract}
We introduce relative versions of non-nef loci and restricted base loci, and investigate a relative version of Boucksom--Broustet--Pacienza's conjecture that these two loci coincide.
We prove the conjecture holds for normal integral schemes projective and flat over a Dedekind domain of characteristic zero with mild singularities, such as BCM-regularity.
This result is a mixed characteristic analogue of results of Cacciola--Di Biagio in characteristic zero and Sato in positive characteristic.
To prove this result, we show that BCM-regularity is preserved under perturbations by divisors with small vanishing order.
\end{abstract}

\maketitle
\markboth{R.~IWANE}{NON-NEF LOCI IN MIXED CHARACTERISTIC}

\tableofcontents

\section{Introduction}

Let $X$ be a normal projective variety over an arbitrary field $k$ and $D$ be an $\R$-Cartier $\R$-divisor on $X$. There are two important loci on $X$ that measure the non-nefness of $D$: the non-nef locus $\NNef(D)$ and the restricted base locus $\Bm(D)$. 

The non-nef locus $\NNef(D)$ is defined as the union of the centers on $X$ of all divisorial valuations $v$ satisfying $v_{\num}(D)>0$. When $D$ is a big Cartier divisor, the numerical vanishing order $v_{\num}(D)$ is defined as
\[v_{\num}(D):=\inf_{m\in\Z_{>0},G\in|mD|} v(G)/m. \]
It is known that $\NNef(D)=\emptyset$ if and only if $D$ is nef when the base field $k$ is uncountable.

The restricted base locus
$\Bm(D)$ is a perturbed version of the stable base locus, defined as
\[\Bm (D) := \bigcup_{A} \bigcap_{m\in\Z_{>0}} \Bs |m(D+A)|, \]
where $A$ runs through all ample $\R$-Cartier $\R$-divisors such that $D+A$ is $\Q$-Cartier and $m$ runs through all positive integers for which $m(D+A)$ is Cartier. As in the case of the non-nef locus, it is known that $\Bm(D)=\emptyset$ if and only if $D$ is nef.

These two loci are expected to coincide.

\begin{conj}[{\cite[Conjecture~2.7]{BBP}}]\label{BBPconj}
Let $X$ be a normal projective variety over a field $k$.
For every $\R$-Cartier $\R$-divisor $D$ on $X$,
\[\NNef(D) = \Bm (D). \]
\end{conj}

In characteristic zero, when $X$ is smooth, Conjecture~\ref{BBPconj} was proved by Nakayama \cite[Chapter~5, 1.9.~Lemma~(1)]{Nak} and an alternative proof via multiplier ideals was given in \cite[Proposition 2.8]{ELMNP}.
Cacciola and Di Biagio \cite[Corollary 4.9]{CdB} generalized the result to the case where $X$ has only klt type singularities.

In positive characteristic, when $X$ is smooth, Musta\c{t}\v{a} \cite[Theorem~6.2]{Mus} proved Conjecture~\ref{BBPconj} using test ideals, which are regarded characteristic $p>0$ counterparts of multiplier ideals in characteristic zero. Sato \cite[Theorem~1.2]{Sato} generalized the result to the case where $X$ has only strongly $F$-regular singularities and $k$ is $F$-finite.
Murayama \cite[Theorem~C]{Mur2} proved that Sato's result holds even if $k$ is not $F$-finite.

In this paper, we introduce relative versions of non-nef loci and restricted base loci (Section~\ref{sec3}), and we investigate the relative version of Conjecture~\ref{BBPconj}.

\begin{conj}\label{mainconj}
Let $f:X \rightarrow S$ be a projective surjective morphism of integral excellent schemes with $X$ normal.
For every $\R$-Cartier $\R$-divisor $D$ on $X$,
\[\NNef(D/S) = \Bm (D/S). \]
(See Section~\ref{sec3} for the precise definitions of both sides.)
\end{conj}

In the case where $\dim X\leq 2$, Conjecture~\ref{mainconj} follows directly from formal properties of restricted base loci and non-nef loci (Corollary~\ref{2dimcase}). Thus, it suffices to consider the case where $\dim X\geq 3$.

When $S$ is of equal characteristic, we obtain the following results by relativizing the proofs of \cite{CdB} and \cite{Sato}.

\begin{thm}[Theorem~\ref{maintheoremzero} and \ref{maintheoremp}]\label{maintheoremequal}
Let $S$ be an integral excellent $\Q$-scheme admitting a dualizing complex (resp.~an integral and $F$-finite $\F_p$-scheme), and let $f:X\rightarrow S$ be a projective surjective morphism from a normal integral scheme.
If $X$ has only klt type (resp.~strongly $F$-regular) singularities except for finitely many closed points, then Conjecture~\ref{mainconj} holds for $f$.
\end{thm}

Recently, the theory of test ideals in mixed characteristic has been developed, and there are several approaches to defining test ideals (e.g., \cite{MS}, \cite{MST}, \cite{BMP}, \cite{HLS}, and \cite{BMP2}).
Ma and Schwede \cite{MS} defined a version of test ideals detecting BCM-regular singularities, which can be viewed as an analogue of klt singularities in equal characteristic zero and strongly $F$-regular singularities in characteristic $p>0$.
On the other hand,
Hacon, Lamarche, and Schwede \cite{HLS} defined a version of test ideals satisfying uniform global generation, which is known for multiplier ideals in characteristic zero. As an application, they essentially proved Conjecture~\ref{mainconj} in \cite[Theorem~7.7]{HLS} when $X$ is a regular scheme projective over a complete Noetherian local domain of mixed characteristic.
Bhatt et al. \cite{BMP2} defined a version of test ideals satisfying uniform global generation and detecting klt and BCM-regular singularities after localiation and completion at a point.

Our main theorem generalizes the result of \cite[Theorem~7.7]{HLS} to the case where $X$ has mild singularities and $S$ is not necessarily complete local.
We say that $(x\in X)$ is of BCM-regular type if there exists an effective $\Q$-divisor $\Delta_x$ on $\Spec \sO_{X,x}$ such that $(\widehat{\sO_{X,x}},\widehat{\Delta_x})$ is BCM-regular.
For example, local log-regular rings, which can be viewed as a generalization of toric singularities to the mixed characteristic case, are of BCM-regular type (see Section~\ref{sec:BCMtestideal}).

\begin{mainthm}[Theorem~\ref{maintheorem}]\label{mainthm}
Let $S$ be a Dedekind scheme of characteristic zero, and
let $f:X \rightarrow S$ be a projective surjective morphism from a normal integral scheme. Suppose that every non-closed point $x\in X$ satisfies one of the following conditions:
\begin{enumerate}
\item $(x\in X)$ is of equal characteristic zero and of klt type.
\item $(x\in X)$ is of mixed characteristic and of BCM-regular type.
\end{enumerate}
Then Conjecture~\ref{mainconj} holds for $f$.
\end{mainthm}

As a corollary, we obtain the following:

\begin{cor}[Corollary~\ref{3dimkltcase}]\label{mainthmklt}
Let $f:X \rightarrow S=\Spec \Z[1/30]$ be a projective surjective morphism from a normal integral scheme. Suppose that $\dim X=3$ and $X$ has only klt type singularities except for finitely many closed points.
Then Conjecture~\ref{mainconj} holds for $f$.
\end{cor}

We sketch the proof of the Main Theorem.
Since the formation of non-nef loci and restricted base loci commutes with localization of the base (Proposition~\ref{bcnnef} and \ref{bcbaselocus}),
the proof reduces to the case where $S$ is the spectrum of a DVR of equal characteristic zero or mixed characteristic.

The equal characteristic zero case follows by Theorem~\ref{maintheoremequal}.
In the mixed characteristic case, following strategies similar to those in \cite{CdB} and \cite{Sato}, we compare three loci on $X$: the non-nef locus, the restricted base locus, and the zero locus of asymptotic test ideals.
Here, we use the version of test ideals introduced by \cite{BMP2}. We prove the following:

\begin{thm}[Theorem~\ref{maintheoremtestideal}]
Let $X$ be a normal integral scheme projective and flat over a DVR $V$ of mixed characteristic.
For a log $\Q$-Gorenstein pair $(X,\Delta)$ and a big $\Q$-Cartier $\Q$-divisor on $X$, we have
\[
\NNef(D)\setminus\Zero(\tau(X,\Delta))
=\Bm(D)\setminus\Zero(\tau(X,\Delta))
=\bigcup_{m\in\Z_{>0}} \Zero(\tau(X,\Delta,\|mD\|))\setminus\Zero(\tau(X,\Delta)),
\]
where $\tau(X,\Delta)$ denotes the test ideal
and $\tau(X,\Delta,\|mD\|)$ denotes the asymptotic test ideal
of \cite{BMP2}, and $\Zero(\mathcal{I})$ denotes the zero locus of an ideal sheaf $\mathcal{I}\subset\sO_{X}$.
\end{thm}

The containment $\NNef(D)\subset
\Bm(D)$ (Proposition~\ref{easycontainment}) follows from formal properties.
By the uniform global generation of asymptotic test ideals (Proposition~\ref{uggmixed}),
we also have
$\Bm(D)\subset\bigcup_{m\in\Z_{>0}} \Zero(\tau(X,\Delta,\|mD\|))$
(Proposition~\ref{uggmixedcontainment}).
The comparison between the zero locus of asymptotic test ideals and the non-nef locus reduces to the uniform stability of BCM-regularity under perturbations by divisors, which is a mixed characteristic analogue of \cite[Corollary~3.14]{Sato}.
This was essentially proved by \cite[the paragraph following Definition~5.3.1]{CLM} for $\Q$-Gorenstein rings.

\begin{prop}[{Proposition~\ref{stability}, cf.~\cite[Definition~5.3.1]{CLM}}]\label{introstability}
Let $(R,\m)$ be a complete normal Noetherian local domain of residue characteristic $p>0$, and let $\Delta$ be an effective $\Q$-divisor on $\Spec R$ such that $(R,\Delta)$ is BCM-regular.
Then there exists $\delta\in\R_{>0}$ such that $(R,\Delta +t\cdot\divi(r))$ is BCM-regular for every $0\neq r\in \m$ and every $t\in\Q_{\geq 0}$ satisfying $t\cdot \ord_{\m}(r)<\delta$.
\end{prop}

This proposition follows from \cite[Lemma~2.5.1 and Corollary~2.5.3]{CLM} (Lemma~\ref{replaceBCM}), which are essentially due to Gabber.

In Section~\ref{sec:BCMalpha}, as a topic of independent interest, we investigate (a variant of) the supremum of the possible values of $\delta$ in Proposition~\ref{introstability}, which we call the local BCM-alpha invariant.
In characteristic $p>0$,
this invariant coincides with the local $F$-alpha invariant, which was recently introduced by \cite{Pande2} and \cite{TY} as a positive characteristic analogue of Tian's $\alpha$-invariant \cite{Tian} in K-stability theory.

\begin{small}
\begin{acknowledgement}
I would like to express his deepest gratitude to his supervisor Professor Shunsuke Takagi for his encouragement, valuable advice, and suggestions.
I would like to thank Hirotaka Onuki and Yuki Morita for their comments in the graduate seminar.
I would also like to thank Tatsuki Yamaguchi, Ryo Ishizuka, and Shou Yoshikawa for helpful conversations.
I am grateful to Kenta Sato for explaining the proof of Proposition~\ref{nnefempty} in the absolute setting.
This research was supported by the WINGS-FMSP program at the University of Tokyo.
\end{acknowledgement}
\end{small}

\section{Preliminaries}

\subsection{Conventions}\label{sec:conv}

\begin{enumerate}
\item 
We say that a scheme $X$ is excellent if $X$ is Noetherian and has a open cover by spectra of excellent rings.
\item For a scheme $X$ and an ideal sheaf $\mathcal{I}\subset\sO_{X}$,
set $\Zero(\mathcal{I}):=\Supp(\sO_{X}/\mathcal{I})$.
\item We say that a morphism of schemes $X\rightarrow S$ is projective if there exist $N\in\Z_{\geq 0}$ and a closed immersion $X\hookrightarrow \PP_{S}^{N}$ over $S$.
\item Let $X$ be a normal integral scheme admitting a dualizing complex $\omega_{X}^{\bullet}$, and let $\Delta$ be an effective $\Q$-divisor on $X$. Take an associated canonical divisor $K_X$ with the smallest nonzero cohomology of $\omega_{X}^{\bullet}$.
We say that $(X,\Delta)$ is a log $\Q$-Gorenstein pair if $K_X +\Delta$ is $\Q$-Cartier. When $X=\Spec R$ for a ring $R$, we say that $(R,\Delta)$ is a log $\Q$-Gorenstein pair.
We say that $X$ (resp.~$R$) is $\Q$-Gorenstein if $(X,0)$ (resp.~$(R,0)$) is log $\Q$-Gorenstein. These definitions are independent of the choice of $\omega_{X}^{\bullet}$ and $K_X$.
\item
Let $X$ be a normal integral scheme of finite type over a regular, integral, and excellent base scheme, and let $(X,\Delta)$ be a log $\Q$-Gorenstein pair.
For the definition of klt pairs $(X,\Delta)$, 
see \cite[Section~2.1]{Kol} or \cite[Section~2.5]{BMP}.
This definition is consistent with Definition~\ref{kltdef}.
\item
For a flat morphism $f:Y\rightarrow X$ of normal integral Noetherian schemes and an $\R$-divisor $D$ on $X$, we often write the reflexive pullback of $D$ to $Y$ by $D|_{Y}$.
When $Y=\Spec R$ (resp.~$Y=\Spec \sO_{X,x}$), we write $D|_{\Spec R}$ (resp.~$D|_{\Spec\sO_{X,x}}$) by $D|_{R}$ (resp.~$D_x$).
When $(R,\m)$ is a Noetherian local ring and $Y=\Spec \Rhat  \rightarrow \Spec R =X$ is the $\m$-adic completion, we write $D|_{\Spec\Rhat}$ by $\widehat{D}$.
\item For a Noetherian local ring $(R,\m)$ and $r\in R$, the vanishing order of $r$ along $\m$ is defined as
$\ord_{\m}(r):= \max \{ n\in\Z_{\geq 0} \mid r\in\m^{n} \}$, and the normalized order of $r$ along $\m$ is defined as
$\nord_{\m}(r):= \max \{ n\in\Z_{\geq 0} \mid r\in\overline{\m^{n}} \}$, where $\overline{\m^{n}}$ denotes the integral closure of $\m^{n}$ for each $n$.
\end{enumerate}

\subsection{Valuations}
Let $X$ be an integral Noetherian scheme, and
let $K(X)$ denote the function field of $X$.

\begin{defn}
A real valuation $v$ of $K(X)$ is a map $v:K(X)\rightarrow\R\cup\{\infty\}$ such that for every $f,g\in K(X)$, the following hold:
\begin{enumerate}
    \item $v(fg)=v(f)+v(g)$.
    \item $v(f+g)\geq \min\{v(f),v(g)\}$.
    \item $v(f)=\infty$ if and only if $f=0$.
\end{enumerate}
\end{defn}

\begin{defn}
Let $v$ be a real valuation of $K(X)$.
\begin{enumerate}
\item
$\sO_v :=\{f\in K(X) \mid v(f)\geq 0 \}$,
$\m_v :=\{f\in K(X) \mid v(f)>0 \}$,
and
$k(v):=\sO_v / \m_v$.
\item
We say that $v$ is centered on $X$ if there exists $z\in X$ such that $v$ is non-negative on $\sO_{X,z}$. 
In this situation, the induced map $\phi:\Spec \sO_v \rightarrow \Spec \sO_{X,z} \rightarrow X$ is independent of the choice of $z$,
and $c_X(v):=\phi(\m_v)$ is called the center of $v$ on $X$.
\item
We say that $v$ is a divisorial valuation over $X$
if $v$ is centered on $X$ and $\sO_v$ is a DVR essentially of finite type over $X$.
\end{enumerate}
\end{defn}

\begin{rem}\label{excval}
Suppose that $X$ is excellent.
Since the normalizations of schemes of finite type over $X$ is finite, a real valuation
$v$ on $K(X)$ is a divisorial valuation over $X$ if and only if there exist a projective birational morphism $\mu:Y\rightarrow X$ from a normal integral scheme and a prime divisor $\Gamma$ on $Y$ such that $v=\lambda\cdot\ord_\Gamma$ for some $\lambda\in\R_{>0}$.
\end{rem}

\subsection{Relative positivity properties of divisors}

In this section, we recall some definitions of relative positivity properties of divisors.

Let $f: X \rightarrow S$ be a projective surjective morphism of integral Noetherian schemes with $X$ normal.
Let $\sim_S$ (resp.~$\equiv_S$) denote the $f$-linear equivalence (resp.~$f$-numerical equivalence).
We write the generic fiber of $f$ by $X_{\eta}$.

\begin{defn}[{\cite[Definition 2.2 and 2.9]{Keeler},\cite[Definition 3.1]{Ushiro}}]
Let $L$ be a line bundle on $X$.
We say that $L$ is:
\begin{enumerate}
\item $f$-ample
if there exists an affine open cover $\{ U_{\lambda} \}$ of $S$ such that $L|_{f^{-1}(U_{\lambda})}$ is ample for every $\lambda$;
\item $f$-nef
if $(L.C)\geq 0$ for every irreducible curve $C\subset X$ that is closed in $X$ and satisfies that $f(C)$ is a closed point of $S$;
\item $f$-semiample if $f^{*}f_{*}L^{\otimes m}\rightarrow L^{m}$ is surjective for some $m\in\Z_{>0}$;
\item $f$-$\Q$-effective if $f_{*}L^{\otimes m}\neq 0$ for some $m\in\Z_{>0}$;
\item $f$-big if $f_{*}(L^{\otimes m}\otimes A^{-1})\neq 0$
for some $m\in\Z_{>0}$ and some $f$-ample line bundle $A$.
\end{enumerate}
These notions extend to $\Q$-Cartier $\Q$-divisors $D$ by requiring that some positive multiple $L:=\sO_X(mD)$ satisfies the corresponding properties.
\end{defn}

\begin{rem}
By \cite[Theorem 3.6]{Keeler},
${N}^{1}(X/S)_{\R} := (\Pic(X)/\equiv_S) \otimes_{\Z} \R$ is a finite dimensional $\R$-vector space.
Furthermore, by \cite[Proposition 2.10 and Theorem 3.9]{Keeler}, it is known that $f$-ampleness is an $f$-numerically invariant property and the real cone generated by numerical classes of $f$-ample line bundles is a non-empty open convex cone of ${N}^{1}(X/S)_{\R}$ with respect to the Euclidean topology.
\end{rem}

\begin{rem}
By \cite[Remark~3.2 and Proposition~3.5]{Ushiro}, the following are equivalent.
\begin{enumerate}
\item
$L$ is $f$-big.
\item
For every $f$-ample line bundle $A$, $f_{*}(L^{\otimes m}\otimes A^{-1})\neq 0$ for some $m\in\Z_{>0}$.
\item
$L|_{X_{\eta}}$ is big.
\end{enumerate}
\end{rem}

\begin{defn}[{\cite[Section 2]{Les}}]
Let $D$ be an $\R$-Cartier $\R$-divisor.
We say that $D$ is:
\begin{enumerate}
\item
$f$-ample if its numerical class lies in the real cone generated by numerical classes of $f$-ample line bundles;
\item
$f$-nef if $(D.C)\geq 0$ for every irreducible curve $C\subset X$ that is closed in $X$ and satisfies that $f(C)$ is a closed point of $S$;
\item
$f$-big if $D|_{X_{\eta}}$ is a big $\R$-Cartier $\R$-divisor on $X_{\eta}$;
\item
$f$-pseudo-effective if $D|_{X_{\eta}}$ is a pseudo-effective $\R$-Cartier $\R$-divisor on $X_{\eta}$.
\end{enumerate}
\end{defn}

\begin{rem}
In the situation above, it follows straightforward that the following are equivalent.
\begin{enumerate}
\item
$D$ is $f$-pseudo-effective.
\item
For every $f$-ample $\R$-Cartier $\R$-divisor $A$, $D+A$ is $f$-big.
\item
The numerical class of $D$ can be expressed as the limit of a sequence of numerical classes of $f$-$\Q$-effective $\Q$-Cartier $\Q$-divisors.
\end{enumerate}
\end{rem}

\subsection{Multiplier ideals}\label{secmult}
In this section, we recall the definitions of multiplier ideals in our setting and summarize some properties of them that we use later.

Let $X$ be a normal, integral, and excellent scheme admitting a dualizing complex ${\omega}_{X}^{\bullet}$, and let $\Gamma$ be a $\Q$-Cartier $\Q$-divisor on $X$.
For a proper birational morphism $\pi:Y\rightarrow X$ with $Y$ normal,
we write the exceptional pullback functor for $\pi$ by $\pi^{!}$.
Then ${\omega}_{Y}^{\bullet}:=\pi^{!}{\omega}_{X}^{\bullet}$ is a dualizing complex of $Y$.
By applying Grothendieck duality for $\pi^{!}$ and ${\omega}_{X}^{\bullet}$,
we obtain the canonical map
\[ \bR \pi_{*}{\omega}_{Y}^{\bullet}= \bR \pi_{*}\pi^{!}{\omega}_{X}^{\bullet} \rightarrow {\omega}_{X}^{\bullet}. \]
This map is called the trace map for $\pi$.
By \cite[Proposition~2.18]{BST}, this map induces an injective map
\[ {\Tr}_{\pi}:\pi_{*}{\omega}_{Y}(-\lfloor \pi^{*}\Gamma \rfloor) \hookrightarrow {\omega}_{X}(- \lfloor\Gamma\rfloor), \]
where ${\omega}_{X}$ (resp.~${\omega}_{Y}$) denotes the smallest nonzero cohomology of ${\omega}_{X}^{\bullet}$ (resp.~${\omega}_{Y}^{\bullet}$).

\begin{defn}[{\cite[Definition~2.26]{BST}}]\label{defmult}
The multiplier module $\J(\omega_X,\Gamma)$ is defined as
\[ \J(\omega_X,\Gamma):=\bigcap_{\pi:Y\rightarrow X} \Image ({\Tr}_{\pi}:\pi_{*}{\omega}_{Y}(-\lfloor \pi^{*}\Gamma \rfloor) \hookrightarrow {\omega}_{X}(- \lfloor\Gamma\rfloor)),\]
where $\pi$ runs through all proper birational morphisms $Y\rightarrow X$ with $Y$ normal.
For a log $\Q$-Gorenstein pair $(X,\Delta)$, the multiplier ideal $\J(X,\Delta)$ is defined by
\[ \J(X,\Delta):=\J(\omega_X, K_X +\Delta), \]
where $K_X$ is a canonical divisor associated with ${\omega}_{X}$. Since the choice of $K_X$ determines the inclusion ${\omega}_{X}\subset K(X)$, the multiplier ideal $\J(X,\Delta)$ can be regarded as an ideal of $\sO_X$. Note that the resulting ideal $\J(X,\Delta)\subset\sO_X$ is independent of the choice of $K_X$.
\end{defn}

In the rest of this section, we assume that $X$ is a $\Q$-scheme or a separated scheme of dimension at most $3$.
By \cite{Tem} and \cite{CP}, there exists a log resolution of $(X,\Gamma)$,
i.e. a projective birational morphism $\pi:Y\rightarrow X$ from a regular integral scheme such that ${\pi}^{-1}(\Supp\Gamma)\cup\mathrm{Ex}(\pi)$ is a divisor with simple normal crossings.
It is well-known that the multiplier module and the multiplier ideal are computed by any single log resolution, that is, the intersection in Definition~\ref{defmult} stabilizes at any single log resolution.
However, we record its proof here, adapting it in our situation.

\begin{prop}\label{singleresol}
With notation as above, for a log resolution $\pi:Y\rightarrow X$ of $(X,\Gamma)$,
\[ \J(\omega_X,\Gamma)= \Image ({\Tr}_{\pi}:\pi_{*}{\omega}_{Y}(-\lfloor \pi^{*}\Gamma \rfloor) \hookrightarrow {\omega}_{X}(- \lfloor\Gamma\rfloor)). \]
\end{prop}

\begin{proof}
Since any log resolutions can be dominated by a third one, and trace maps are compatible with compositions, it is enough to show that for $\mu:Z\rightarrow Y$ such that $\pi\circ\mu:Z\rightarrow X$ is a log resolution of $(X,\Gamma)$, the trace map
\[ \Tr_{\mu}:\mu_{*} \omega_Z (-\lfloor\mu^{*}\pi^{*}\Gamma \lfloor) \hookrightarrow \omega_Y (-\lfloor\pi^{*}\Gamma\rfloor)\]
is an isomorphism. Set $\Gamma^{\prime}=\pi^{*}\Gamma$ By Chow's Lemma, we can further assume that $\mu$ is projective.
We also may assume that $\lfloor\Gamma^{\prime}\rfloor=0$.
Define a relative canonical divisor $K_{Z/Y}$ supported on the exceptional set of $\mu$ as \cite[Definition~2.4]{Kol}.
Then
${\mu}_{*} \omega_{Z} \hookrightarrow \omega_Y$
and
${\mu}_{*} \sO_{Z}(K_{Z/Y}) \hookrightarrow \sO_{Y}$ 
can be identified since these are reflexification maps.
The statement follows from \cite[Corollary~2.11]{Kol} and the same arguments as those in \cite[Lemma~9.2.19]{Laz}.
\end{proof}

\begin{defn}\label{kltdef}
For a log $\Q$-Gorenstein pair $(X,\Delta)$ and $x\in X$, we say that $(X,\Delta)$ is klt at $x$ if $\J(X,\Delta)_x =\sO_{X,x}$. This definition only depends on $(\sO_{X,x},\Delta_x)$ and is consistent with the definition in Section~\ref{sec:conv}(3).
We say that $(x\in X)$ is of klt type if there exists a $\Q$-divisor $\Delta_x\geq 0$ on $\Spec\sO_{X,x}$ such that $(\sO_{X,x},\Delta_x)$ is klt at $x$.
\end{defn}

\begin{prop}[{\cite[Proposition~9.5.13]{Laz}}]\label{regklt}
With notation as above, we further assume that $X$ is a regular $\Q$-scheme. For every point $x\in X$,
if $\ord_{\m_{x}}(\Delta)<1$, then $(X,\Delta)$ is klt at $x\in X$.
\end{prop}

\begin{proof}
By localizing at $x$, we may assume that $X=\Spec R$ is the spectrum of an excellent regular local ring $(R,\m,k)$.
Since $k$ is infinite, by \cite[Proposition~8.5.7(3)]{HS}, there exists $h\in\m\setminus {\m}^2$ such that $\ord_{\m}(\Delta)=\ord_{\m/hR}(\Delta|_{R/hR})$.
By the restriction formula \cite[Theorem~A.1]{JM}, $\J(R/hR,\Delta|_{R/hR})\subset\J(R,\Delta)\cdot R/hR$.
Hence, the statement follows by the induction on $\dim R$.
\end{proof}

Let $S$ be an affine, integral, and excellent $\Q$-scheme admitting a dualizing complex ${\omega}_{S}^{\bullet}$, and let $f:X\rightarrow S$ be a projective surjective morphism.
Take a dualizing complex ${\omega}_{X}^{\bullet}:= f^{!}{\omega}_{S}^{\bullet}$ and an associated canonical divisor $K_X$.

\begin{defn}[{\cite[Definition~11.1.2]{Laz}}]
With notation as above, for a $\Q$-effective $\Q$-Cartier $\Q$-divisor $D$ on $X$, the asymptotic multiplier ideal $\J(X,\Delta,\|D\|)$ is defined as
\[ \J(X,\Delta,\|D\|):=\bigcup_{\substack{m\in\Z_{>0} \\ \text{$mD$ Cartier}}} \sum_{G\in |mD|} \J(X,\Delta + \frac{1}{m}G). \]
\end{defn}

\begin{prop}[{\cite[Proposition~9.4.26]{Laz}}]\label{uggmult}
With notation as above, let $A$ be a globally generated ample line bundle, $M$ be a line bundle on $X$, and $n$ be a positive integer such that $n\geq\max_{s\in S}\dim X_{s}$.
\begin{enumerate}
\item
If $M-(K_X +\Delta)$ is big and nef,
then $\J(X,\Delta)\otimes\sO_{X}(nA+M)$ is globally generated.
\item
If $M-(K_X +\Delta)-D$ is big and nef, then $\J(X,\Delta,\|D\|)\otimes\sO_{X}(nA+M)$ is globally generated.
\end{enumerate}
\end{prop}

\begin{proof}
(2) follows from (1) by the definition of $\J(X,\Delta,\|D\|)$. 
We prove (1).
We follow the same arguments as those in the absolute case \cite[Proposition~9.4.26]{Laz}.
By relative Castelnuovo-Mumford regularity (e.g. simple relativization of \cite[section 1.8]{Laz} or  \cite{Ooishi}), it is enough to show that
\[ R^{i}f_{*}(\J(X,\Delta)\otimes\sO_{X}(nA+M-iA))=0 \]
for every $i\in\Z_{>0}.$ For $i>\max_{s\in S}\dim X_{s},$ this is trivial. Assume that $i\leq\max_{s\in S}\dim X_{s}.$

Let $\pi:Y\rightarrow X$ be a proper birational morphism from a regular integral scheme such that ${\pi}^{-1}(\Supp(K_X +\Delta)\cup\Supp M\cup\Supp A)\cup\mathrm{Ex}(\pi)$ is a divisor with simple normal crossings.
By Proposition~\ref{singleresol} and Grothendieck-Leray spectral sequence, it is enough to show that
\begin{gather*}
R^i(f\circ\pi)_{*} \omega_{Y}(-\lfloor{\pi}^{*}(K_X +\Delta)\rfloor + {\pi}^{*}(M+(n-i)A)) \\
= R^i(f\circ\pi)_{*} \omega_{Y}(\lceil  {\pi}^{*}(M-(K_X +\Delta))+ (n-i){\pi}^{*}A \rceil)
=0,
\end{gather*}
and
\[ R^j \pi_{*} \omega_{Y}(-\lfloor{\pi}^{*}(K_X +\Delta) \rfloor ) =0,\]
for every $j\in\Z_{>0}.$

The first vanishing follows from
a generalization of Kawamata-Viehweg vanishing by \cite[Theorem A]{Mur}
since ${\pi}^{*}(M-(K_X +\Delta))+ (n-i){\pi}^{*}A$ is a nef and big divisor whose support is simple normal crossings.
The second vanishing is equivalent to
\[ R^j \pi_{*} \omega_{Y}(\lceil {\pi}^{*}(lA- (K_X +\Delta)) \rceil) =0\]
for some $0\ll l\in\Z$.
Since ${\pi}^{*}(lA- (K_X +\Delta))$ is a $\pi$-nef and $\pi$-big divisor whose support is simple normal crossings,
this also follows from \cite[Theorem A]{Mur}.
\end{proof}

\subsection{Test ideals in characteristic $p>0$}\label{sectest}
In this section, we recall the definitions of test ideals in characteristic $p>0$ in our setting and summarize some properties of them that we use later.

Let $X$ be a normal, integral, and $F$-finite Noetherian $\F_p$-scheme admitting a dualizing complex ${\omega}_{X}^{\bullet}$.
Take a canonical divisor $K_X$ associated with the smallest nonzero cohomology of $\omega_{X}^{\bullet}$.
Let $(X,\Delta)$ be a log $\Q$-Gorenstein pair.
We first recall the definition of the test ideal of $(X,\Delta)$.
We refer the readers to \cite{BSTZ}.

By applying Grothendieck duality for the $e$-th iteration of the absolute Frobenius that is denoted by $F^e$,
we have 
\[{\Hom}(F^e_* \sO_X ((1- p^e) (K_X +\Delta)),\sO_X)\cong F^e_* \Hom(\sO_X ((1- p^e)\Delta),\sO_X).\]
Let ${\phi}_{e,\Delta}:F^{e}_{*}\sO_{X}((1-q)(K_X +\Delta))\rightarrow\sO_{X}$ be the morphism corresponding to the natural inclusion $\sO_X ((1- p^e)\Delta)\subset\sO_X$.

When $(q-1)(K_X +\Delta)$ is Cartier for some $p$-power $q=p^{e_0} >1$,
the test ideal $\tau(X,\Delta)$ is defined to be the smallest nonzero coherent ideal $\mathcal{I}\subset\sO_X$ satisfying
\[{\phi}_{e,\Delta}(F^{e}_{*}(\mathcal{I}\cdot\sO_{X}((1-p^e)(K_X +\Delta))))\subset\mathcal{I}\]
for all positive multiples $e$ of $e_0$, that is known to exist.
In general, the test ideal $\tau(X,\Delta)$ is defined as
\[ \tau(X,\Delta):=\sum_{\Delta^\prime}\tau(X,\Delta^\prime), \]
where $\Delta^\prime$ runs through all $\Q$-divisors on $X$ such that $\Delta^\prime \geq\Delta$ and $(q-1)(K_X +\Delta^\prime)$ is Cartier for some $p$-power $q>1$.

For a $\Q$-effective $\Q$-Cartier $\Q$-divisor $D$ on $X$, the asymptotic test ideal $\tau(X,\Delta,\|D\|)$ is defined as
\[ \tau(X,\Delta,\|D\|):=\bigcup_{\substack{m\in\Z_{>0} \\ \text{$mD$ Cartier}}} \sum_{G\in |mD|} \tau(X,\Delta + \frac{1}{m}G). \]

\begin{defn}
With notation as above,
we say that $(X,\Delta)$ is strongly $F$-regular at $x\in X$ if $\tau(X,\Delta)_x =\sO_{X,x}$. This definition only depends on $(\sO_{X,x},\Delta_x)$.
We say that $x\in X$ is strongly $F$-regular if there exists a $\Q$-divisor $\Delta_x \geq 0$ on $\Spec\sO_{X,x}$ such that $(\sO_{X,x},\Delta_x)$ is strongly $F$-regular at $x$.
\end{defn}

Let $S$ be an affine, integral, and $F$-finite $\F_p$-scheme, and let $f:X\rightarrow S$ be a projective surjective morphism.
It is known that $S$ is excellent and admits a dualizing complex ${\omega}_{S}^{\bullet}$.
Take a dualizing complex ${\omega}_{X}^{\bullet}:= f^{!}{\omega}_{S}^{\bullet}$ and an associated canonical divisor $K_X$.

\begin{prop}[{\cite[Theorem~4.1]{Mus} and \cite[Proposition~4.1]{Sato}}]\label{uggtest}
With notation as above, let $A$ be a globally generated ample line bundle, $M$ be a line bundle on $X$, and $n$ be a positive integer such that $n>\max_{s\in S}\dim X_{s}$.
\begin{enumerate}
\item
If $M-(K_X +\Delta)$ is ample,
then $\tau(X,\Delta)\otimes\sO_{X}(nA+M)$ is globally generated.
\item
If $M-(K_X +\Delta)-D$ is ample, then $\tau(X,\Delta,\|D\|)\otimes\sO_{X}(nA+M)$ is globally generated.
\end{enumerate}
\end{prop}

\begin{proof}
(2) follows from (1) by the definition of $\tau(X,\Delta,\|D\|)$. 
We prove (1).
We follow the same arguments as those in the proof of the absolute case \cite[Theorem~4.1]{Mus} or \cite[Proposition~4.1]{Sato}.
By enlarging $\Delta$, we may assume that $(q-1)(K_X +\Delta)$ is Cartier for some $p$-power $q=p^e >1$.
Then ${\phi}_{e,\Delta}(\tau(X,\Delta)\cdot\sO_{X}((1-q)(K_X +\Delta)))=\tau(X,\Delta)$.
By tensoring this with $\sO_{X}(nA+M)$, we have a surjection
\[ F^{e}_{*}(\tau(X,\Delta)\otimes\sO_{X}((1-q)(K_X +\Delta)+q(nA+M)))\twoheadrightarrow\tau(X,\Delta)\otimes\sO_{X}(nA+M). \]
Thus, the proof reduces to show that $F^{e}_{*}\tau(X,\Delta)\otimes\sO_{X}((1-q)(K_X +\Delta)+q(nA+M))$ is globally generated.
By relative Castelnuovo-Mumford regularity (e.g. simple relativization of \cite[section 1.8]{Laz} or  \cite{Ooishi}), it is enough to show that
\[{H}^{i}(F^{e}_{*}(\tau(X,\Delta)\otimes\sO_{X}((1-q)(K_X +\Delta)+q(nA+M)))\otimes{\sO}_{X}(-iA))=0\]
for every $i\in\Z_{>0}$.
For $i>\max_{s\in S}\dim X_{s},$ this is trivial. Assume that $i\leq\max_{s\in S}\dim X_{s}$.
Note that $i<n$ and
\[ (1-q)(K_X +\Delta)+q(nA+M)-qiA=(K_X +\Delta)+q(M-(K_X +\Delta))+q(n-i)A. \] 
Since $M-(K_X +\Delta)$ and $A$ is ample,
this vanishing follows from relative Fujita's vanishing theorem \cite[Theorem~1.5]{Keeler} by taking any sufficiently large multiple of $e$.
\end{proof}

\subsection{Test ideals in mixed characteristic}

\subsubsection{\bf{BCM-test ideals}}\label{sec:BCMtestideal}

Let $(R,\m,k)$ be a complete normal Noetherian local domain of residue characteristic $p>0$.
Let $R^+$ denote the absolute integral closure of $R$ (i.e., the integral closure of $R$ in a fixed algebraic closure of the fraction field of $R$).

We first recall the definition of BCM-test ideals introduced in \cite{MS} from the perspective of  perfectoid big Cohen-Macaulay algebras.
See \cite[Section 2.1]{MS} for the definition of perfectoid big Cohen-Macaulay $R^+$-algebras.
By \cite[Theorem~2.5]{BMP}, it is known that the $p$-adic completion $\Rphat$ of $R$ is a perfectoid big Cohen-Macaulay $R^+$-algebra.

\begin{defn}[{\cite[Definition~6.2, 6.9, and Proposition~6.10]{MS}}]\label{BCMtestideal}
With notation as above,
set $d:=\dim R$.
Let $(R,\Delta)$ be a log $\Q$-Gorenstein pair, and let $E_R(k)$ be the injective hull of $k$ as an $R$-module.
Fix a canonical divisor $K_R \geq 0$.
Write
$K_R +\Delta=\frac{1}{m}\divi(f)$ for some $0\neq f\in R$ and some $m\in\Z_{>0}$.
Take a representing element $f^{\frac{1}{m}}\in R^+$.
Let $B$ be a perfectoid big Cohen-Macaulay $R^{+}$-algebra and
let $\phi:H_{\m}^{d}(R(K_R)) \rightarrow H_{\m}^{d}(B)$ denote the morphism such that
the composition of the natural morphism $H_{\m}^{d}(R) \rightarrow H_{\m}^{d}(R(K_R))$ and $\phi$ is identical to the morphism induced by the multiplication by $f^{\frac{1}{m}}$.

The BCM-test ideal of $(R,\Delta)$ with respect to $B$ is defined as
\[ \tau_{B}(R,\Delta) 
:= {\Image( E_R(k) \cong H_{\m}^{d}(R(K_R)) \xrightarrow{\phi} H_{\m}^{d}(B))}^{\vee} \subset R, \]
where ${(-)}^{\vee}$ denotes Matlis duality.
The BCM-test ideal of $(R,\Delta)$ is defined as
\[ \taubcm(R,\Delta) := \bigcap_{B} \tau_B (R,\Delta) \subset R.\]
We say
that $(R,\Delta)$ is $\BCM_B$-regular if $\tau_B (R,\Delta)=R$
and
that $(R,\Delta)$ is BCM-regular if $\taubcm(R,\Delta)=R$.
These definitions are independent of the choices of $K_R$ and $f^\frac{1}{m}\in R^+$.
\end{defn}

It is known that BCM-test ideals satisfies analogs of good properties of multiplier ideals in equal characteristic zero, such as restriction formula, summation formula, Skoda's theorem, and stability under small perturbations of divisors.
We refer the readers to \cite{MS}, \cite[Section 8.1]{BMP2}, and \cite{MST} for properties of BCM-test ideals.

Next, we define the notion of ``BCM-regular type'' for local rings that are not necessarily complete. For complete local rings, this has been already introduced by
\cite[Definition~5.3.1(d)]{CLM}.

\begin{defn}[{cf.~\cite[Definition~5.3.1(d)]{CLM}}]
Let $(A,\m_A,k)$ be a normal excellent local domain of residue characteristic $p>0$.
We say that $A$ is of BCM-regular type if there exists a $\Q$-divisor $\Delta_A \geq 0$ such that $(A,\Delta_A)$ is log $\Q$-Gorenstein and the $\m_A$-adic completion
$(\widehat{A},\widehat{\Delta_A})$ is BCM-regular.
When $A=\sO_{X,x}$ for a scheme $X$ and a point $x\in X$, we say that $(x\in X)$ is of BCM-regular type.
\end{defn}

We note that the underlying rings of local log-regular rings are of BCM-regular type.
This was proved by \cite[Proposition~5.3.5]{CLM} for complete local rings.

By the structure theorem of local log-regular rings due to Kato \cite[Theorem~2.22]{INS},
we can adapt the following as the definition of log-regularity.
See \cite[Section~2.2.1]{INS} for the precise definition of local log-regular rings and some properties of monoids.

\begin{thm}[{\cite[Theorem~2.22]{INS}}]\label{logregstructure}
Let $A$ be a normal excellent local domain of residue characteristic $p>0$.
Then $A$ is the underlying ring of a local log-regular ring
if and only if
there exist a fine, sharp, and saturated monoid $\mathcal{Q}$, a morphism of monoids $\alpha:\mathcal{Q}\rightarrow A$, and $r\in\Z_{\geq 0}$ satisfying one of the following:
\begin{enumerate}
\item 
$A$ is of characteristic $p>0$ and
there exists a commutative diagram
\[
  \xymatrix{
    \mathcal{Q} \ar[r] \ar[d]_{\alpha} & k[[\mathcal{Q}\oplus\N^{r}]] \ar[d]^{\phi} \\
    A \ar[r] & \widehat{A},
  }
\]
where $\phi$ is an isomorphism.
\item
$A$ is of mixed characteristic and
there exists a commutative diagram
\[
  \xymatrix{
    \mathcal{Q} \ar[r] \ar[d]_{\alpha} & {C}_{k}[[\mathcal{Q}\oplus\N^{r}]] \ar[d]^{\phi} \\
    A \ar[r] & \widehat{A},
  }
\]
where $\ker\phi =(\theta)$ for some $\theta\in {C}_{k}[[\mathcal{Q}\oplus\N^{r}]]$ whose constant term is $p$.
\end{enumerate}
\end{thm}

\begin{prop}[{cf.~\cite[Proposition~5.3.5]{CLM}}]
With notation as above, suppose that $A$ is the underlying ring of a local log-regular ring.
Then $A$ is of BCM-regular type.
\end{prop}

\begin{proof}
We only prove the mixed characteristic case since the positive characteristic case can be proven by similar arguments.
By \cite[Proposition~5.3.5]{CLM}, there exists a $\Q$-divisor $\Deltahat\geq 0$ such that $(\widehat{A},\Deltahat)$ is BCM-regular.
It suffices to show that such $\Deltahat$ can be taken as the pullback of some $\Q$-divisor $\Delta\geq 0$ on $\Spec A$.
For this, we recall the construction of $\Deltahat$ in \cite[Proposition~5.3.5]{CLM}, adapting it in our situation.

Set ${A}^{\prime}:=C_k[[\mathcal{Q}\oplus \N^{r}]]$.
By the assumptions on $\mathcal{Q}$,
the Grothendieck group $\mathcal{Q}^{\gp}$ is a free abelian group.
Set $\R\mathcal{Q}:=\mathcal{Q}^{\gp}\otimes_{\Z}\R$ and let $\R_{\geq0}\mathcal{Q}\subset\R\mathcal{Q}$ be the rational polyhedral cone generated by $\mathcal{Q}$.
By Gordan's lemma \cite[Proposition~6.1.2(a)]{BH}, $\mathcal{Q}=\mathcal{Q}^{\gp}\cap\R_{\geq0}\mathcal{Q}\subset\R\mathcal{Q}$.
Take the first lattice point $v_1,\cdots,v_t$ of extremal rays of the dual cone ${(\R_{\geq0}\mathcal{Q})}^{\vee}$.
For $1\leq i \leq t$, let ${D}_{i}^{\prime}$ be the Weil divisor on $\Spec {A}^{\prime}$ corresponding to the prime ideal generated by $\q_i:=\{ x\in\mathcal{Q} \mid v_i(x)>0 \}\subset\mathcal{Q}$.
Take an element $q\in\mathcal{Q}$ contained in the relative interior of $\R_{\geq0}\mathcal{Q}$.
Then $\divi_{A^\prime}(q)=\sum_{i=1}^{t}b_i {D}_{i}^{\prime}$ for some $b_i >0$.
By applying \cite[Lemma~2.5]{robinson} to $C_k [[\mathcal{Q}]]$, we can take the canonical divisor $K_{{A}^{\prime}}=-\sum_{i=1}^{t}{D}_{i}^{\prime}$.
Define the $\Q$-divisor $\Delta^{\prime}\geq 0$ on $\Spec A^{\prime}$ as
$\Delta^{\prime}:=\sum_{i=1}^{t}(1-\epsilon b_i){D}_{i}^{\prime}$, where $\epsilon\in\Q_{>0}$ is sufficiently small.
Then $\divi(p-f)$ and $\Delta^{\prime}$ have no common components and
\[ K_{{A}^{\prime}}+\divi(p-f)+\Delta^{\prime}=-\epsilon\sum_{i=1}^{t}b_i {D}_{i}^{\prime}=-\epsilon\divi(q) \]
is $\Q$-Cartier.
Let $\Deltahat\geq 0$ be the different of $K_{{A}^{\prime}}+\divi(p-f)+\Delta^{\prime}$ along $\divi(p-f)=\Spec \widehat{A}$.
Note that we can restrict $D_i^{\prime}$ to $\Spec \widehat{A}$ as a Weil divisor since $D_i^{\prime}$ is normal.
Then $\Deltahat=\sum_{i=1}^{t}(1-\epsilon b_i){D}_{i}$, where ${D}_{i}:={D_i^{\prime}}|_{\widehat{A}}$,
i.e., ${D}_{i}$ is corresponding to the reflexive hull of $\alpha(\q_i) \widehat{A}$.

It follows from \cite[Proposition~3.10]{ishiro} that $\alpha(\q_i)A$ is a non-zero prime ideal of $A$.
Hence,
$0< \mathrm{ht}(\alpha(\q_i) A)
\leq \mathrm{ht}{( \alpha(\q_i) \widehat{A} )}^{**}=1$.
Thus, we can define $\Delta\geq 0$ as a $\Q$-linear combination of the prime ideals $\alpha(\q_i)A$ of height one.
\end{proof}

\subsubsection{\bf{A unified theory of test ideals}}\label{sectestmixed}

In \cite{BMP2}, the authors introduced a version of test ideals for normal integral schemes of finite type over a DVR that satisfies both good local and global properties and that unifies some versions of test ideals previously proposed.
See \cite{BMP2} for the precise definition of thier test ideals.

\begin{prop}[{\cite[Definition 7.13, 7.20, Corollary~7.24, and Theorem B]{BMP2}}]\label{uniftest}
Let $(V,\m)$ be a DVR of mixed characteristic $(0,p>0)$.
For every triple $(X,\Delta,\ideala^{t})$, where $X$ is a normal integral scheme of finite type and flat over $V$, $(X,\Delta)$ is a log $\Q$-Gorenstein pair, $\ideala$ is a coherent ideal of $\sO_{X}$, and $t\in\Q_{\geq 0}$,
there exists a coherent ideal $\tau(X,\Delta,\ideala^{t})\subset\sO_{X}$ satisfying the following:

\begin{enumerate}
\item For an open subset $U$ of $X$ flat over $V$,
$\tau(X,\Delta,\ideala^{t})|_{U}=\tau(U,\Delta|_{U},{(\ideala\sO_U)}^{t}).$
\item $\tau(X,\Delta,\ideala^{t})[1/p]=\J(X[1/p],\Delta|_{X[1/p]},(\ideala\sO_{X[1/p]})^{t})$.
\item For every point $x\in X_{\m}$,
\[ \tau(X,\Delta,\ideala^{t})\cdot \widehat{\sO_{X,x}}=
\taubcm(\widehat{\sO_{X,x}},\widehat{\Delta_x},(\ideala\cdot\widehat{\sO_{X,x}})^{t}). \]
\item (Uniform global generation)
Suppose that $X$ is projective over $V$. For a line bundle $M$ on $X$ such that $M- K_X -\Delta$ is big and nef,
$\tau(X,\Delta)\otimes\sO_{X}(nA+M)$ is globally generated for every $n\in\Z_{>0}$ with $n\geq\dim X_{\m}$.
\item The restriction formula, the summation formula, Skoda's theorem, and stability under small perturbations of divisors are satisfied.
\end{enumerate}
\end{prop}

\begin{rem}
In the situation of (4), when $V$ is complete, $\tau(X,\Delta)\otimes\sO_{X}(nA+M)$ is globally generated by a submodule of $H^0 (X,\tau(X,\Delta)\otimes\sO_{X}(nA+M))$ called $+$-stable sections.
The global generation of $\tau(X,\Delta)\otimes\sO_{X}(nA+M)$ follows since $\tau(X,\Delta)$ is compatible with the base change of $V$ to the $\m$-adic completion $\widehat{V}$ by \cite[Remark~7.15]{BMP2}.
\end{rem}

Next, we recall the asymptotic version of these test ideals essentially introduced in \cite[Remark~8.42]{BMP2} and its uniform global generation property.

\begin{defn}[{\cite[Remark 8.42]{BMP2}}]
Let $(V,\m)$ be a DVR of mixed characteristic
and $X$ a normal integral scheme projective and flat over $V$.
For a log $\Q$-Gorenstein pair $(X,\Delta)$ and a big $\Q$-Cartier $\Q$-divisor $D$ on $X$,
the asymptotic test ideal of $(X,\Delta)$ associated to $D$ is defined as
\[ \tau(X,\Delta,\|D\|):=\bigcup_{\substack{m\in\Z_{>0} \\ \text{$mD$ Cartier}}} \sum_{G\in |mD|} \tau(X,\Delta + \frac{1}{m}G). \]
\end{defn}

\begin{prop}[{\cite[Remark 8.42]{BMP2}}]\label{uggmixed}
With notation as above,
let $A$ be a globally generated and ample line bundle on $X$.
If $M$ is a line bundle on $X$ such that $M- (K_X +\Delta)-D$ is big and nef, then $\tau(X,\Delta, \|D\|)\otimes\sO_{X}(nA+M)$ is globally generated for every $n\in\Z_{>0}$ with $n\geq\dim X_{\m}.$
\end{prop}

\begin{proof}
This follows from Proposition~\ref{uniftest}(4) and the definition of $\tau(X,\Delta, \|D\|)$.
\end{proof}

\section{Relative versions of base loci and non-nef loci}\label{sec3}

\subsection{Relative base loci}

In this section, we define relative versions of augmented base loci and restricted base loci and investigate their properties.

\begin{defn}
Let $f: X \rightarrow S$ be a projective surjective morphism of integral Noetherian schemes with $X$ normal.
\begin{enumerate}
\item 
For a Cartier divisor $D$ on $X$,
the base locus of $D$ over $S$ is defined as
\[\Bs|D/S| := \Supp \Coker (f^{*}f_{*} \sO_{X}(D) \rightarrow \sO_{X}(D)). \]
\item
For a $\Q$-Cartier $\Q$-divisor $D$ on $X$,
the stable base locus of $D$ over $S$ is defined as
\[\B (D/S) := \bigcap_{\substack{m \in \mathbb{Z}_{>0} \\ \text{$mD$ Cartier}}} \Bs|mD/S|.\]
\item
For an $\R$-Cartier $\R$-divisor $D$ on $X$,
the augmented base locus of $D$ over $S$ is define as
\[\Bp (D/S) := \bigcap_{A}  \B (D-A/S), \]
where $A$ runs through all $f$-ample $\R$-Cartier $\R$-divisors such that $D-A$ is $\Q$-Cartier, and the restricted base locus of $D$ over $S$ is defined as
\[\Bm (D/S) := \bigcup_{A}  \B (D+A/S), \]
where $A$ runs through all $f$-ample $\R$-Cartier $\R$-divisors such that $D+A$ is $\Q$-Cartier.
\end{enumerate}
\end{defn}

\begin{rem}\label{baselocusfacts}
Fix a Euclidean norm $|\cdot|$ on ${N}^{1}(X/S)_{\R}$. Then the relative base loci defined above have the same properties as in the absolute case 
\cite[Proposition 1.4 - Corollary 1.6, Example 1.8, 1.9, Lemma 1.14 - Example 1.16, and Proposition 1.19 - 1.21]{ELMNP}. 
In particular, the following hold.
\begin{enumerate}
\item For every $\Q$-Cartier $\Q$-divisor $D$ and every sufficiently divisible $m \in \mathbb{Z}_{>0}$,
\[\B (D/S) = \Bs|mD/S|. \]
\item For every $\R$-Cartier $\R$-divisor $D$ and every $f$-ample $\R$-Cartier $\R$-divisor $A$ such that $|A| \ll 1$ and $D-A$ is $\Q$-Cartier,
\[\Bp (D/S) = \B (D-A/S). \]
\item For every $\R$-Cartier $\R$-divisor $D$ and
every sequence of $f$-ample $\R$-Cartier $\R$-divisors $(A_i)_{i \in \mathbb{Z}_{>0}}$ satisfying $|A_i| \xrightarrow{i\to\infty} 0$,
\[\Bm (D/S) = \bigcup_{i\in \mathbb{Z}_{>0}}  \Bp (D+ A_i /S)
=\bigcup_{i\in \mathbb{Z}_{>0}}  \Bm (D+ A_i /S). \]
\end{enumerate}
\end{rem}

\begin{rem}
By Remark~\ref{baselocusfacts}, $\B (D/S)$ and $\Bp (D/S)$ are closed subsets of $X$, and $\Bm (D/S)$ is a countable union of closed subsets of $X$.
We do not consider any scheme structures on them.
\end{rem}

\begin{rem}
Suppose that $S$ is affine. For any Cartier divisor $D$ on $X$,
\[\Bs|D/S| = \bigcap_{s \in {H}^{0}(X,D)} \Supp \Coker (\sO_{X} \xrightarrow{s} \sO_{X}(D)). \]
Thus, $\Bp (D/S)$ and $\Bm (D/S)$ do not depend on the choice of the base $S$ since ampleness also does not (\cite[Proposition 2.15]{Keeler}).
Hence, we write $\Bp (D/S)$ and $\Bm (D/S)$ simply as $\Bp (D)$ and $\Bm (D)$, respectively.
\end{rem}

First, we show that the relative augmented base locus and the relative restricted base locus commute with flat base change.

\begin{prop}\label{bcbaselocus}
Let $f: X \rightarrow S$ be a projective morphism from a normal integral scheme to an integral Noetherian scheme.
Let $g:S'\rightarrow S$ be a flat morphism of integral Noetherian schemes such that $X':=X\times_{S}S'$ is normal (e.g., an open immersion or a localization at any point of $S$).
For every $\R$-Cartier $\R$-divisor $D$ on $X$,
\begin{equation*}
\begin{aligned}
{g'}^{-1}(\Bp(D/S)) = \Bp({g'}^{*}D/S'), \\
{g'}^{-1}(\Bm(D/S)) = \Bm({g'}^{*}D/S'),
\end{aligned}
\end{equation*}
where $g':X' \rightarrow X$ is the natural morphism.
\end{prop}

\begin{proof}
Since taking global sections commutes with flat base change, ${g'}^{-1}(\Bs|D/S|)=\Bs|{g'}^{*}D/S'|$ for every Cartier divisor $D$ on $X$.
Thus, the statements follow from Remark~\ref{baselocusfacts} and the fact that relative ampleness is preserved under base change.
\end{proof}

Next, we summarize the relations between the variants of relative base loci and the relative positivity properties.

\begin{prop}
Let $f: X \rightarrow S$ be a projective surjective morphism of integral Noetherian schemes with $X$ normal.

For any $\Q$-Cartier $\Q$-divisor $D$ on $X$,
\begin{enumerate}
\item $ \B (D/S) = \emptyset$ if and only if $D$ is $f$-semiample, and
\item $ \B (D/S) \subsetneq X$ if and only if $D$ is $f$-$\Q$-effective.
\end{enumerate}

For any $\R$-Cartier $\R$-divisor $D$ on $X$,
\begin{enumerate}[resume]
\item $ \Bp (D/S) = \emptyset$ if and only if $D$ is $f$-ample,
\item $\Bp (D/S) \subsetneq X$ if and only if $D$ is $f$-big,
\item $ \Bm (D/S) = \emptyset$ if and only if $D$ is $f$-nef, and
\item $\Bm (D/S) \subsetneq X$ if and only if $D$ is $f$-pseudo-effective.
\end{enumerate}
\end{prop}

Note that $\Bp (D/S) \subsetneq X$ (resp.~$\Bm (D/S) \subsetneq X$)
means that $\Bp (D/S)$ (resp.~$\Bm (D/S)$) does not contain the generic point of $X$.

\begin{proof}
(1), (2), and (3) are straightforward from the definition. (Note that the reducedness of $X$ is used in (2). See \cite[Proposition 2.17]{Ushiro}.)
(5) follows from (3) and the Nakai criterion \cite[Proposition~2.10 and 2.11]{Keeler}.
Let $X_{\eta}$ denotes the generic fiber of $f$.
(4) follows from the fact that $\Bp (D/S) \subsetneq X$ is equivalent to $\Bp (D|_{X_{\eta}}) \subsetneq X_{\eta}$ by Proposition~\ref{bcbaselocus}. (6) follows in a similar manner.
\end{proof}

Lastly, we show that, even in the relative setting, $\Bp(D/S)$ has no isolated closed points.
In the absolute setting, this was proved in \cite[Proposition~1.1]{ELMNP2} using Zariski's theorem, which asserts a line bundle $L$ is semiample if $\Bs|L|$ is a finite set of closed points.
Moret-Baily \cite[Theorem~1.2]{Moret} generalized Zariski's theorem to the relative setting.

\begin{lem}[The relative Fujita-Zariski theorem, {\cite[Theorem~1.2]{Moret}}]\label{fujzar}
Let $f: X \rightarrow S=\Spec R$ be a proper morphism of Noetherian schemes, and let $L$ be a line bundle on $X$.
If ${L|}_{\Bs|L|}$ is ample, then $L$ is semiample. (Note that the ampleness of ${L|}_{\Bs|L|}$ does not depend on the scheme structure of $\Bs|L|$.)
\end{lem}

\begin{cor}\label{clpt}
Let $f: X \rightarrow S$ be a projective surjective morphism of integral Noetherian schemes with $X$ normal, and let $D$ be an $\R$-Cartier $\R$-divisor on $X$. Then the following hold.
\begin{enumerate}
\item $\Bp(D/S)$ has no isolated closed points. 
\item If $x\in\Bm(D/S)$ is a closed point of $X$, then there exists an irreducible closed subset $Z\subset X$ such that $\{x\}\subsetneq Z \subset \Bm(D/S)$.
\end{enumerate}
\end{cor}

\begin{proof}
We may assume that $S$ is affine by Proposition~\ref{bcbaselocus}.
(1) follows from Remark~\ref{baselocusfacts}(2), Lemma~\ref{fujzar}, and the same arguments as in \cite[Proposition 1.1]{ELMNP2}. (2) follows from (1) and the definition of $\Bm(D/S)$.
\end{proof}

\subsection{Relative non-nef loci}

In this section, we introduce and investigate relative versions of asymptotic and numerical vanishing orders along divisorial valuations and non-nef loci.
In the absolute setting, they were introduced by \cite[Chapter~3]{Nak}, \cite{ELMNP}, and \cite[Section~2.2]{BBP}, and in the relative setting, for projective morphisms of complex varieties, they were introduced by
\cite[Chapter~3, Section~4]{Nak} and \cite[Section~3]{LX}.

We first define relative asymptotic vanishing orders along prime divisors.
There are two different ways to define them, and we show in Lemma~\ref{asyord} that they coincide. This has been already shown in \cite[Lemma~3.2]{LX} for smooth complex varieties.

\begin{defn}[{cf.~\cite[Definition 3.1]{LX} and \cite[Chapter~3, Section 4.a]{Nak}}]
Let $f: X \rightarrow S$ be a projective surjective morphism of integral excellent schemes with $X$ normal.
Let $D$ be a $\Q$-Cartier $\Q$-divisor on $X$, and let $\Gamma\subset X$ be a prime divisor on $X$. We define $\ord_{\Gamma}(X/S,\|D\|)$ and $\ord_{\Gamma}(\|D\|,X/S)$ by
\begin{equation*}
\begin{gathered}
\ord_{\Gamma}(X/S,\|D\|) := \inf_{m\in\Z_{>0}} \dfrac{1}{m} \min \left\{ \ord_{\Gamma}D' \middle| mD \sim_{S} D' \geq 0 \right\}, \\
\ord_{\Gamma}(\|D\|,X/S) := \inf_{m\in\Z_{>0}} \dfrac{1}{m} \max \left\{ l \in \Z_{\geq 0} \middle| f_{*}\sO_{X}(mD-l\Gamma) \xrightarrow{\sim} f_{*}\sO_{X}(mD) \right\},
\end{gathered}
\end{equation*}
where $m$ runs through all positive integers such that $mD$ is Cartier.
If there are no positive integers $m$ such that
$\{ D' \mid mD \sim_{S} D' \geq 0 \}\neq\emptyset$, we set $\ord_{\Gamma}(X/S,\|D\|) := +\infty$, and if the maximum in the latter equation does not exist, we set
$\ord_{\Gamma}(\|D\|,X/S):=+\infty$.
\end{defn}

\begin{rem}\label{ordrem}
When $S$ is affine, for every $\Q$-effective $\Q$-Cartier $\Q$-divisor $D$ on $X$, 
\[ \ord_{\Gamma}(\|D\|,X/S) = \inf_{m\in\Z_{>0}} \dfrac{1}{m} \min \left\{ \ord_{\Gamma}D' \middle| mD \sim D' \geq 0 \right\}.  \] 
\end{rem}

\begin{lem}\label{ordaffine}
With notation as above, let $S^0 \subset S$ be an open subset with $f(\Gamma)\cap S^0 \neq \emptyset$. Set $X^0 := f^{-1}(S^0)$, $D^0 := D|_{X^0}$, $f^0:=f|_{X^0}$, and ${\Gamma}^{0} := \Gamma \cap X^0$.
Then the following hold.
\begin{enumerate}
\item $ \ord_{\Gamma^{0}}(X^0 / S^0 , \|D^0\|)=\ord_{\Gamma}(X/S,\|D\|)$.
\item $\ord_{\Gamma^{0}}(\|D^0\| , X^0 / S^0)=\ord_{\Gamma}(\|D\|,X/S).$
\end{enumerate}
\end{lem}

\begin{proof}
First, we prove (1).
Fix any $m\in\Z_{>0}$ such that $mD$ is Cartier. Take an effective Cartier divisor $D'$ on $X^0$ such that $mD^{0}\sim_{S^0} D'$ and 
\[ 
\ord_{\Gamma^{0}}D'=
\min \left\{ \ord_{\Gamma^{0}}D'' \middle| mD^{0} \sim_{S^0} D'' \geq 0 \right\}.
\] 
There exists a Cartier divisor $E^0$ on $S^0$ such that $mD^{0}+ (f^0)^{*} E^0 \sim D'\geq 0$.
There exists an open subset $S^1 \subset S^0$ such that $E^0 |_{S^1}$ is principal, $S^1 \cap f(\Gamma)\neq \emptyset$, and $S\setminus S^1$ is pure of codimension one.
Let $E$ be a principal divisor on $S$ defined by $E^0 |_{S^1}$, and let $\overline{D'}$ be the closure of $D'$ in $X$. Then 
\[ mD+f^{*}E \sim \overline{D'}+F \]
for some Weil divisor $F$ supported on $X\setminus f^{-1}(S^1)$.
Note that $F$ may not be effective.
Let $E'\geq 0$ be any sufficiently large Cartier divisor on $S$ such that $f(F)\subset \Supp E' \subset S\setminus S^1$.
Then 
\[ mD+f^{*}(E+E')\sim\overline{D'}+(F+f^{*}E')\]
and $F+f^{*}E'\geq 0$ is supported on $X\setminus X^1$.
Thus, \[ mD\sim_{S}\overline{D'}+(F+f^{*}E')\geq 0\] and
\begin{equation*}
\begin{aligned}
\min \left\{ \ord_{{\Gamma}^{0}}D'' \middle| mD \sim_{S^0} D'' \geq 0 \right\}
=\ord_{\Gamma^{0}}D' 
&=\ord_{\Gamma}(\overline{D'}+(F+f^{*}E')) \\
& \geq
\min \left\{ \ord_{\Gamma}D'' \middle| mD \sim_{S} D'' \geq 0 \right\}.
\end{aligned}
\end{equation*}
The reverse inequality follows by definition.

Next, we prove (2). When $X$ is regular, this was proved by \cite[Chapter~3, 4.1.~Lemma]{Nak}. We follow similar arguments.
For any $i\in \Z_{\geq 0}$,
we consider the exact sequence
\[0 \rightarrow \sO_{X}(mD-(i+1)\Gamma) \rightarrow \sO_{X}(mD-i\Gamma) \rightarrow \mathcal{F} \rightarrow 0. \]
By \cite[Lemma 2.60(3)]{Kol}, $\mathcal{F}$ is a torsion-free $\sO_{\Gamma}$-module, since the first and the second term are reflexive. By taking pushforward, we obtain the exact sequence
\[0 \rightarrow f_{*}\sO_{X}(mD-(i+1)\Gamma) \rightarrow f_{*}\sO_{X}(mD-i\Gamma) \rightarrow f_{*}\mathcal{F}. \]
Since $f_{*}\mathcal{F}$ is a torsion-free $\sO_{f(\Gamma)}$-module,
 $f_{*}\sO_{X}(mD-(i+1)\Gamma) \xrightarrow{\sim} f_{*}\sO_{X}(mD-i\Gamma)$ if and only if $f_{*}\sO_{X}(mD-(i+1)\Gamma) \rightarrow f_{*}\sO_{X}(mD-i\Gamma)$ is isomorphic on $S^0$.
\end{proof}

\begin{lem}\label{asyord}
With notation as above,
$\ord_{\Gamma}(X/S,\|D\|) = \ord_{\Gamma}(\|D\|,X/S)$.
\end{lem}

\begin{proof}
By Lemma~\ref{ordaffine}, we may assume that $S$ is affine. 
By Remark~\ref{ordrem}, it is enough to show that for every $m\in \Z_{>0}$ such that $mD$ is Cartier,
\[ \min \{ \ord_{\Gamma}D' \mid mD \sim_{S} D' \geq 0 \} \geq \min \{ \ord_{\Gamma}D' \mid mD \sim D' \geq 0 \}. \]
Take an effective Cartier divisor $E$ on $X$ attaining the minimum of the left side.
Since shrinking $S$ does not change the right side by the proof of Lemma~\ref{ordaffine}(2),
\begin{equation*}
\ord_{\Gamma}E \geq \min \left\{ \ord_{\Gamma}D' \middle| mD \sim D' \geq 0 \right\}. \qedhere   
\end{equation*}
\end{proof}

Next, we define relative asymptotic vanishing orders along divisorial valuations.

\begin{defn}[{\cite[Definition~3.6]{LX}}]\label{asyvan}
Let $f: X \rightarrow S$ be a projective surjective morphism of integral excellent schemes with $X$ normal. Let $D$ be a $\Q$-Cartier $\Q$-divisor on $X$, and let $v$ be a divisorial valuation over $X$.
Take a projective birational morphism $\mu :Y \rightarrow X$ from a normal integral scheme, a prime divisor $\Gamma$ on $Y$, and $\lambda\in\R_{>0}$ such that $v=\lambda\cdot\ord_{\Gamma}$ (Remark~\ref{excval}).
We define the asymptotic vanishing order $v(\| D/S \|)$ of $D$ along $v$ as
\[ v(\| D/S \|) := \lambda\cdot\ord_{\Gamma}(Y/S,\|\mu^{*}D\|) = \lambda\cdot\ord_{\Gamma}(\|\mu^{*}D\|,Y/S). \]
This definition is independent of the choice of $Y$ and $\Gamma$ by Lemma~\ref{ordaffine} and Remark~\ref{ordrem}.
When $S$ is affine, we write $v(\| D/S \|)$ by $v(\| D\|)$ since it is independent of the choice of $S$ by Remark~\ref{ordrem}.
\end{defn}

\begin{lem}\label{bcasyvan}
With notation as above, let $g:S' \rightarrow S$ be either an open immersion satisfying $c_{X}(v)\in S'$ or a localization $S'=\Spec \sO_{S,\xi} \rightarrow S$ at a point $\xi \in \overline{c_{X}(v)}$.
Then
\[ v(\| D/S \|)= v(\| {g'}^{*}D/ S' \|), \]
where $g':X \times_{S} S' \rightarrow X$ is the natural morphism.
\end{lem}

\begin{proof}
When $g$ is an open immersion, this follows from \ref{ordaffine}.
Suppose that $g$ is a localization.
With notation as in Definition~\ref{asyvan},
\begin{align*}
v(\| D/S \|) &= \ord_{\Gamma}(\|\mu^{*}D\|,Y/S) \\
&= \inf_{m\in\Z_{>0}} \dfrac{1}{m} \max \left\{ l \in \Z_{\geq 0} \middle| f_{*}\sO_{Y}(m\mu^{*}D-l\Gamma) \xrightarrow{\sim} f_{*}\sO_{Y}(m\mu^{*}D) \right\}.
\end{align*}
The statement follows since the isomorphism in the right side can be checked at $\xi$ by the proof of Lemma~\ref{ordaffine}(2).
\end{proof}

Next, we define relative numerical vanishing orders of big $\Q$-Cartier $\Q$-divisors.

\begin{defn}
With notation as above, let $D$ be an $f$-big $\Q$-Cartier $\Q$-divisor on $X$.
We define the numerical vanishing order of $D$ along $v$ as
\begin{equation*}
v_{\num}(D/S)=
\inf_{S^0 \subset S} 
\inf \left\{ v(D') \middle|
\begin{gathered}
 \text{$D'$ is a $\Q$-Cartier $\Q$-divisor on $f^{-1}(S^0)$} \\
 \text{such that $D|_{f^{-1}(S^0)} \equiv_{S^0} D' \geq 0$}
\end{gathered} 
\right\}, 
\end{equation*}
where $S^0$ runs through all open neighborhoods of $c_{X}(v)$ in $S$.
By definition, $v_{\num}(D/S)$ only depends on the $f$-numerical equivalence class of $D$.
\end{defn}

Although the infimum is taken over $S^0$ in the definition above, we show in the following proposition that $v_{\num}(D/S)$ can be computed by any $S^0$.

\begin{prop}[{cf.~\cite[Chapter~3, 1.4. Lemma~(2)]{Nak}}]
With notation as above, 
\[ v(\| D/S \|)=v_{\num}(D/S), \]
and for any affine open neighborhood $S^0$ of $c_{X}(v)$ in $S$,
\begin{equation*}
v_{\num}(D/S)= \inf \left\{ v(D') \middle|
\begin{gathered}
 \text{$D'$ is a $\Q$-Cartier $\Q$-divisor on $f^{-1}(S^0)$} \\
 \text{such that $D|_{f^{-1}(S^0)} \equiv_{S^0} D' \geq 0$}
\end{gathered} 
\right\}. 
\end{equation*}
\end{prop}

\begin{proof}
By Lemma~\ref{ordaffine}, it is enough to show $v(\| D/S \|) = \inf \left\{ v(D') \mid D \equiv_{S} D' \geq 0 \right\}$ when $S$ is affine.
We follow the same arguments as those in \cite[Chapter~3, 1.4.~Lemma~(2)]{Nak}.

Take any ample Cartier divisor $A$ on $X$. We first claim that
\[ v(\| D \|) = \lim_{\Q_{>0} \ni \varepsilon \rightarrow +0} v(\| D+\varepsilon A \|). \]
Since $S$ is affine and $D$ is big, there are $\delta \in \Q_{>0}$ and an effective $\Q$-Cartier $\Q$-divisor $\Delta$ on $X$ such that $D \sim_{\Q} \delta A+\Delta$.
Then for any $\varepsilon \in \Q_{>0}$,
\[(1+\varepsilon)v(\| D+\varepsilon A \|) \leq (1+\varepsilon)v(\| D \|) \leq v(\| D+\varepsilon\delta A \|)+\varepsilon v(\Delta).\]
By taking the limit $\varepsilon \rightarrow +0$, the claim is obtained.

For any effective $\Q$-Cartier $\Q$-divisor $D'$ on $X$ such that $D' \equiv_{S} D$,
\[ v(\| D+\varepsilon A \|) \leq v(\| D+\varepsilon A-D' \|)+ v(D')=v(D') \]
since $D+\varepsilon A-D'$ is ample.
Thus, by taking the limit $\varepsilon \rightarrow +0$, we obtain $v(\|D\|) \leq v(D')$. Therefore, 
\[ v(\|D\|) \leq \inf \left\{ v(D') \mid D \equiv_{S} D' \geq 0 \right\}.\]
The reverse inequality follows by definition.
\end{proof}

\begin{rem}
When $X$ is regular and $S$ is the spectrum of an algebraically closed field or a complete Noetherian local domain of residue characteristic $p>0$, this proposition was proved via asymptotic multiplier and test ideals in \cite[Corollary~2.7]{ELMNP}, \cite[Theorem~6.1]{Mus}, and \cite[Theorem~7.6]{HLS}.
\end{rem}

In the following proposition, let $\Bigcone(X/S)_{\Q}$ (resp.~$\Bigcone(X/S)_{\R}$) denote the cone of $f$-numerical equivalence classes of $f$-big $\Q$-Cartier $\Q$-divisors (resp.~$f$-big $\R$-Cartier $\R$-divisors).

\begin{prop}[{cf.~\cite[Theorem A]{ELMNP}}]\label{conti}
With notation as above, let $S^0$ be an affine open neighborhood of $c_{X}(v)$ in $S$.
The function
\[ v(\|-/S\|)=v_{\num}(-/S): \Bigcone(X/S)_{\Q} \rightarrow \R_{\geq 0} \]
uniquely extends to the continuous function
\[ \Bigcone(X/S)_{\R} \rightarrow \R_{\geq 0}, \]
mapping an $f$-big $\R$-Cartier $\R$-divisor $D$ on $X$ to
\begin{equation*}
\inf \left\{ v(D') \middle|
\begin{gathered}
 \text{$D'$ is a $\R$-Cartier $\R$-divisor on $f^{-1}(S^0)$} \\
 \text{such that $D|_{f^{-1}(S^0)} \equiv_{S^0} D' \geq 0$.}
\end{gathered} 
\right\} 
\end{equation*}
\end{prop}

\begin{proof}
Since $N^1(X/S)_{\R}$ is a finite-dimensional $\R$-vector space by \cite[Theorem~3.6]{Keeler}, the statement can be proved by the same arguments as those in the proofs of \cite[Theorem~A and Lemma~3.3]{ELMNP}.
\end{proof}

We now define relative numerical vanishing orders and relative non-nef loci for pseudo-effective $\R$-Cartier $\R$-divisors.

\begin{defn}\label{numvan}
With notation as above, the function $\Bigcone(X/S)_{\R} \rightarrow \R_{\geq 0}$ defined in Proposition~\ref{conti} is also denoted by $v_{\num}(-/S)$.
For an $f$-big $\R$-Cartier $\R$-divisor $D$ on $X$, we call $v_{\num}(D/S)$ the numerical vanishing order of $D$ along $v$.
For an $f$-pseudo-effective $\R$-Cartier $\R$-divisor $D$ on $X$,
we define the numerical vanishing order $v_{\num}(D/S)$ of $D$ along $v$ as
\[ v_{\num}(D/S) := \sup_{A} v_{\num}(D+A/S)
= \lim_{i \rightarrow \infty } v_{\num}(D+ A_i /S) \in \R_{\geq 0}\cup\{ +\infty \},
\]
where $A$ runs through all $f$-ample $\R$-Cartier $\R$-divisors on $X$ and $(A_i)_{i\in \Z_{>0}}$ is an arbitrary sequence of $f$-ample $\R$-Cartier $\R$-divisors on $X$ satisfying $|A_i| \xrightarrow{i \rightarrow \infty} 0$.
\end{defn}

\begin{defn}
Let $f: X \rightarrow S$ be a projective surjective morphism of integral excellent schemes with $X$ normal. Let $D$ be an $\R$-Cartier $\R$-divisor on $X$.
If $D$ is $f$-pseudo-effective, we define the non-nef locus $\NNef(D/S)$ of $D$ over $S$ as
\[ \NNef(D/S) := \bigcup_{v} \overline{c_{X}(v)}, \]
where $v$ runs through all divisorial valuations over $X$ satisfying $v_{\num}(D/S)>0$.
If $D$ is not $f$-pseudo-effective, we set $\NNef(D/S) := X$.
\end{defn}

Similarly to Definition~\ref{asyvan}, when $S$ is affine, we write $v_{\num}(D/S)$ and $\NNef(D/S)$ by $v_{\num}(D)$ and $\NNef(D)$, respectively.

Lastly, we investigate some properties of relative non-nef loci.

\begin{prop}\label{nneffact}
With notation as above, for any sequence of $f$-ample $\R$-Cartier $\R$-divisors $(A_i)_{i\in \Z_{>0}}$ satisfying $|A_i| \xrightarrow{i \rightarrow \infty} 0$,
\[ \NNef(D/S) = \bigcup_{i\in \Z_{>0}} \NNef(D+A_i/S). \]
\end{prop}

\begin{proof}
This follows from the equality in Definition~\ref{numvan}.
\end{proof}

\begin{prop}\label{bcnnef}
With notation as above, let $g:S' \rightarrow S$ be either an open immersion satisfying $c_{X}(v)\in S'$ or a localization $S'=\Spec \sO_{S,\xi} \rightarrow S$ at a point $\xi \in c_{X}(v)$. Then
\begin{equation*}
\begin{gathered}
v_{\num}(D/S) = v_{\num}({g'}^{*}D/S'), \\
{g'}^{-1}(\NNef(D/S))= \NNef({g'}^{*}D/S'),
\end{gathered}
\end{equation*}
where $g':X \times_{S} S' \rightarrow X$ is the natural morphism.
\end{prop}

\begin{proof}
The former statement follows from Lemma~\ref{bcasyvan} and the equality in Definition~\ref{numvan}.
The latter immediately follows from the former.
\end{proof}

\begin{prop}\label{easycontainment}
With notation as above, $\NNef(D/S)\subset\Bm(D/S)$. 
\end{prop}

\begin{proof}
By Proposition~\ref{bcbaselocus} and \ref{bcnnef}, we may assume that $S$ is affine. Then the statement follows from the same arguments as those in \cite[Lemma~2.6]{BBP}.
\end{proof}

\begin{lem}[{cf.~\cite[Chapter~2, 3.7~Lemma(4)]{Nak}}]\label{goodeffective}
With notation as above, suppose that $S$ is affine and $D$ is big $\Q$-Cartier.
Let $\Gamma$ be a prime divisor on $X$.
If $\ord_{\Gamma}(\|D\|)=0$, then for every ample $\Q$-Cartier $\Q$-divisor $A$ on $X$, there exists an effective $\Q$-Cartier $\Q$-divisor $E$ on $X$ such that $E\sim_{\Q}D+A$ and $\Gamma\not\subset\Supp E$.
\end{lem}

\begin{proof}
We may assume that $A\geq 0$ and $\Gamma \not\subset \Supp A$.
By the same arguments as those in \cite[Lamma~5.5]{HLS}, take an effective $\Q$-Cartier $\Q$-divisor $G$ such that $\ord_{\Gamma}G=1$ and $\Supp(G-\Gamma)\subset \Supp A$.
Suppose that $m$ is a sufficiently large positive integer such that $mA+G$ is ample.
There exists an effective $\Q$-Cartier $\Q$-divisor $E'$ such that $E' \sim_{\Q} D $ and $\delta:=\ord_{\Gamma}E'<\frac{1}{2m}$ since $\ord_{\Gamma}(\|D\|)=0$. Write $E'=E''+\delta\Gamma$, where $E''$ is an effective $\Q$-divisor such that $\Gamma\not\subset\Supp E''$. 
Then 
\[ E''+\delta\left(mA+G\right)+\left(\left(1- \delta m\right)A-\delta(G-\Gamma)\right)=E'+A\sim_{\Q} D+A. \]
Since $\Supp(G-\Gamma)\subset \Supp A$,
\[
\left(1- \delta m\right)A-\delta\left(G-\Gamma\right) \geq \frac{1}{2}A -\frac{1}{2m}\left(G-\Gamma\right)=
\frac{1}{2}\left(A -\frac{1}{m}\left(G-\Gamma\right)\right)\geq 0
\]
for any sufficiently large $m$, and the leftmost divisor does not contain $\Gamma$ in its support.
Since $mA+G$ is ample, we can take $E$ in the statement.
\end{proof}

\begin{prop}[{cf.~\cite[Lemma~3.7(4)]{LX}}]\label{codimone}
With notation as above, the sets of codimension one points of $X$ contained in
$\NNef(D/S)$ and $\Bm(D/S)$ coincide.
\end{prop}

\begin{proof}
We may assume that $S$ is affine and $D$ is big $\Q$-Cartier by Proposition \ref{bcbaselocus}, \ref{bcnnef}, \ref{nneffact}, and Remark~\ref{baselocusfacts}(3).
By Proposition~\ref{easycontainment}, it is enough to show that for every prime divisor $\Gamma$ on $X$, if $\ord_{\Gamma}(\|D\|)=0$, then $\Gamma\not\subset\Bm(D/S)$.
By Lemma~\ref{goodeffective}, for every ample $\Q$-Cartier $\Q$-divisor $A$ on $X$, $\Gamma\not\subset\B(D+A/S)$. Hence, $\Gamma\not\subset\Bm(D/S)$.
\end{proof}

\begin{cor}\label{2dimcase}
With notation as above, if $\dim X\leq 2$, then
\[ \NNef(D/S)=\Bm(D/S). \]
\end{cor}

\begin{proof}
This follows from Proposition~\ref{codimone} and Corollary~\ref{clpt}(2).
\end{proof}

\begin{prop}\label{nnefempty}
With notation as above, if every residue field of $S$ is uncountable, then $\NNef(D/S)=\emptyset$ if and only if $D$ is $f$-nef.
\end{prop}

\begin{proof}
We may assume that $S$ is affine and $D$ is big $\Q$-Cartier by Proposition \ref{bcnnef} and \ref{nneffact}.
The ``if'' part follows from Proposition~\ref{easycontainment}. We prove the ``only if'' part by contradiction. 

Suppose that $D$ is not $f$-nef, i.e., there exists an irreducible closed curve $C \subset X$ such that $\xi:=f(C)$ is a closed point of $S$ and $(D. C) <0$. Take a projective birational morphism $\mu:Y \rightarrow X$ from a normal integral scheme and a prime divisor $\Gamma$ on $Y$ such that $\mu(\Gamma)=C$.
Then $(\mu^{*}D)|_{\Gamma}=\mu|_{\Gamma} ^{*}(D|_{C})$ is not pseudo-effective by \cite[Chapter~2, 5.6.~Lemma(2)]{Nak}.
Here, we use the fact that $\Gamma$ and $C$ are projective varieties over the uncountable field $\kappa(\xi)$.

On the other hand, since $\NNef(D/S)=\emptyset$, we obtain $\ord_{\Gamma}(\| \mu^{*}D \|)=0$.
By Lemma~\ref{goodeffective}, for any ample $\Q$-Cartier $\Q$-divisor $A$ on $Y$, there exists an effective $\Q$-Cartier $\Q$-divisor $E$ such that $E\sim_{\Q}\mu^{*}D +A$ and $\Gamma\not\subset\Supp E$.
Hence, $(\mu^{*}D)|_{\Gamma}$ is pseudo-effective.
This is a contradiction.
\end{proof}

\section{Uniform stability of BCM-regularity under perturbations}\label{sec4}

In this section, we prove uniform stability of BCM-regularity under perturbations of divisors, which is analogous to \cite[Corollary~3.14]{Sato}.
This is essentially proved in \cite[the paragraph following Definition~5.3.1]{CLM} for $\Q$-Gorenstein rings.
Throughout this section,
let $(R,\m,k)$ be a complete normal Noetherian local domain of residue characteristic $p>0$.

\begin{prop}[{cf.~\cite[Corollary~3.14]{Sato} and \cite[Definition~5.3.1]{CLM}}]\label{stability}
Let $(R,\Delta)$ be a BCM-regular pair.
Then there exists $\delta\in\R_{>0}$ such that $(R,\Delta +t\cdot\divi(r))$ is BCM-regular for every $0\neq r\in \m$ and every $t\in\Q_{\geq 0}$ satisfying $t\cdot \ord_{\m}(r)<\delta$.
\end{prop}

For the proof, we need the following lemma.

\begin{lem}[{\cite[Lemma~2.5.1 and Corollary~2.5.3]{CLM}}]\label{replaceBCM}
Let $B$ be a perfectoid big Cohen-Macaulay $R^+$-algebra, and let $v$ be a real valuation of $\overline{K(R)}$ such that $v(\m R^+)>0$.
Suppose that $\{c_j\}_{j\in\Z_{\geq 0}}$ is a sequence of elements of $R^+$ satisfying $v(c_j)\rightarrow 0$ as $j\rightarrow\infty$. 
Then there exists a perfectoid big Cohen-Macaulay $R^+$-algebra $B^{\prime}$ dominating $B$
such that
for every $\eta\in H^{d}_{\m}(B)$, if $c_j \eta=0$ for every $j\in\Z_{\geq 0}$, then the induced morphism $H^{d}_{\m}(B)\rightarrow H^{d}_{\m}(B^{\prime})$ maps $\eta$ to $0$.
\end{lem}

\begin{proof}[Proof of Proposition~\ref{stability}]
Take a divisorial valuation $w$ over $R$ centered on $\m$, and extend $w$ to a real valuation $v$ of $\overline{K(R)}$.
By Izumi's theorem, there exists $C\in\R_{>0}$ such that
\[ \ord_{\m}(r)\geq C\cdot w(r) \]
for every $r\in R$. Thus, it suffices to prove the statement in which $\ord_{\m}$ is replaced by $w$.

Let $\eta$ denote a socle generator of $H^d_{\m}(\omega_R)\cong E_R(k)$.
We define an ideal $I_{\BCM}^{\Delta}$ of $R^+$ as
\[ I_{\BCM}^{\Delta}:=\{c\in R^+ \mid \text{There exists $B$ such that $H^d_{\m}(\omega_R)\xrightarrow{\phi} H^d_{\m}(B) \xrightarrow{c} H^d_{\m}(B)$ maps $\eta$ to $0$} \}. \]
Here, $B$ is a perfectoid big Cohen-Macaulay $R^+$-algebra and $\phi$ is the morphism defined in Definition~\ref{BCMtestideal}.
Note that for every $R$-module $M$, an $R$-linear map $E_R(k)\rightarrow M$ is not injective if and only if it maps $\eta$ to $0$.
Thus, it is enough to show that
\[ v(I_{\BCM}^{\Delta}):=\inf\{v(c) \mid c\in I_{\BCM}^{\Delta} \}>0 \]
since $\delta:=v(I_{\BCM}^{\Delta})$ satisfies the statement for $w$.

Suppose that $v(I_{\BCM}^{\Delta})=0$.
Then there exists a sequence $\{c_j\}_{j\in\Z_{\geq 0}}$ of elements of $R^+$ satisfying the following:
\begin{enumerate}
    \item $v(c_j)\rightarrow 0$ as $j\rightarrow\infty$, and
    \item for every $j\in\Z_{\geq 0}$, there exists a big Cohen-Macaulay $R^+$-algebra $B_j$ such that $H^d_{\m}(\omega_R)\rightarrow H^d_{\m}(B_j) \xrightarrow{c_j} H^d_{\m}(B_j)$ is not injective.
\end{enumerate}
By \cite[Theorem~4.9]{MS}, there exists a perfectoid big Cohen-Macaulay $R^+$-algebra $B$ dominating all $B_j$. Thus, $B_j$ in Condition (2) can be replaced by $B$.
By Lemma~\ref{replaceBCM}, there exists a perfectoid big Cohen-Macaulay $R^+$-algebra $B^{\prime}$ dominating $B$ such that $H^d_{\m}(\omega_R)\rightarrow H^d_{\m}(B)\rightarrow H^d_{\m}(B^{\prime})$ maps $\eta$ to $0$.
This contradicts the BCM-regularity of $(R,\Delta)$.
\end{proof}

By analogy with \cite[Corollary~3.14]{Sato},
one can ask the following question:

\begin{question}
For a log $\Q$-Gorenstein pair $(R,\Delta)$,
does there exist $\delta\in\R_{>0}$ such that
\[ \taubcm(R,\Delta)=\taubcm(R,\Delta+t\cdot\divi(r)) \]
for every $0\neq r\in \m$ and every $t\in\Q_{\geq 0}$ satisfying $t\cdot \ord_{\m}(r)<\delta$?
\end{question}

We are not able to answer this question since a strategy similar to that of the proof above does not seem to work.

\subsection{The local BCM-alpha invariant}\label{sec:BCMalpha}

In this section, we give some estimations of (a variant of) the supremum of $\delta$ in Proposition~\ref{stability}, which we call the local BCM-alpha invariant.
Since we do not use these results in the following sections, the readers who are not interested can skip this section.

\begin{defn}\label{BCMalpha}
Let $(R,\Delta)$ be a BCM-regular pair.
\begin{enumerate}
\item 
The BCM-regular threshold $\bcmt(R,\Delta;r)$ of $(R,\Delta)$ with respect to $0\neq r\in\m$ is defined as
\begin{align*}
\bcmt(R,\Delta;r)
&:=\sup\{t\in\Q_{\geq0} \mid \text{$(R,\Delta+t\cdot\divi(r))$ is BCM-regular} \} \\
&=\inf\{t\in\Q_{\geq0} \mid \text{$(R,\Delta+t\cdot\divi(r))$ is not BCM-regular} \}.
\end{align*}
\item
The local BCM-alpha invariant $\alpha_{\BCM}(R,\Delta)$ of $(R,\Delta)$ is defined as
\[ \alpha_{\BCM}(R,\Delta) := \inf_{0\neq r\in\m} \bcmt(R,\Delta;r)\cdot\nord_{\m}(r), \]
where $\nord_{\m}(r)$ denotes the normalized order of $r$.
\end{enumerate}
When $R$ is of characteristic $p>0$,
the BCM-regular threshold coincides with the $F$-pure threshold, and
the local BCM-alpha invariant $\alpha_{\BCM}(R,\Delta)$ coincides with the local $F$-alpha invariant $\alpha_{F}(R,\Delta)$ recently introduced by Takagi and Yamaguchi \cite[Definition~3.3]{TY}.
\end{defn}

\begin{rem}\label{Falpha}
The local $F$-alpha invariant is a local version of Pande's $F$-alpha invariant \cite[Definition~3.4]{Pande2} \cite[Definition~3.19]{TY}, which is defined for a projective globally $F$-regular pair $(X,\Delta)$ over an $F$-finite field and for an ample $\Q$-Cartier $\Q$-divisor $\Gamma$ on $X$ by
\[ \alpha_{F}(X,\Delta;\Gamma):=\inf_{0\leq D\sim_{\Q}\Gamma}\gfst(X,\Delta;D), \]
where
$\gfst(X,\Delta;D):=\sup\{t\in\Q_{\geq 0}\mid\text{$(X,\Delta+tD)$ is globally $F$-regular} \}$ is the globally $F$-split threshold of $(X,\Delta)$ with respect to $D$.
This invariant is a positive characteristic analogue of Tian's $\alpha$-invariant \cite{Tian} and the global log canonical threshold \cite[Definition~9.15]{BHJ} in K-stability theory for complex varieties.
It is known that all of these invariants are positive real numbers.
\end{rem}

\begin{rem}
The inequalities $0<\alphabcm(R,\Delta)\leq 1$ hold.
The former follows from Propositon~\ref{stability} 
and the latter follows from the similar arguments to \cite[Proposition~3.5(1)]{TY} since BCM-regular pairs are klt \cite[Corollary~6.22]{MS}.
\end{rem}

First, we relate the local BCM-alpha invariant to the $F$-alpha invariant in Remark~\ref{Falpha} by global-to-local reduction.

\begin{lem}[{\cite[Lemma~7.2]{MST}}]\label{localglobal}
Suppose that $R$ is $\Q$-Gorenstein, $\dim R \geq 2$, and the residue field $k=R/\m$ is $F$-finite.
Let $\pi:Y\rightarrow X=\Spec R$ be a blow up of some $\m$-primary ideal $\mathcal{I}$.
Suppose that $\mathcal{I}\cdot\sO_{Y} =\sO_{Y}(-mE)$ for some $m\in\Z_{>0}$, where $E$ is a normal prime $\pi$-exceptional divisor.
Let $(X,\Delta)$ be a log $\Q$-Gorenstein pair.
Let $\Delta_E$ denote the different of $K_Y + E+\pi_{*}^{-1}\Delta$ along $E$.
Suppose that $(E,\Delta_E)$ is globally $F$-regular.
Then $(R,\Delta)$ is BCM-regular.
\end{lem}

\begin{prop}\label{alphabcmf}
In the situation of Lemma~\ref{localglobal},
set
\[ C:=\inf_{0\neq r\in\m} \frac{\nord_{\m}(r)}{\ord_E(r)}. \]
Then $C$ is a positive real number and 
\[ \alphabcm(R,\Delta)\geq C\cdot \alpha_F (E, \Delta_E; -E|_E)>0.\]
\end{prop}

\begin{proof}
By Izumi's theorem \cite[the paragraph preceding Theorem~2.6]{HublSwanson}, $C$ is positive.
The inequality follows from Lemma~\ref{localglobal} and the fact that the different of $K_Y +E+\Delta+t\cdot\pi_{*}^{-1}\divi(r)$ along $E$ equals
$\Delta_E +t\cdot(\pi_{*}^{-1}\divi(r))|_{E}$ and
$\pi_{*}^{-1}\divi(r)\sim_{\Q}-\ord_{E}(r)E$.
\end{proof}

\begin{rem}\label{lowdim}
Suppose that $(R,\Delta)$ be a completion of a klt pair of dimension at most $3$ that is defined over a DVR.
By \cite[Lemma~5.11]{LiuPande}, one can obtain a plt blow-up of $(R,\Delta)$ using MMP.
In \cite[Theorem~7.11]{MST} and \cite[Theorem~5.14]{LiuPande}, by applying Lemma~\ref{localglobal},
it is shown that
the following cases are BCM-regular pairs:
\begin{enumerate}
    \item $\dim R=2$ and $\Delta$ has standard coefficients.
    \item $\dim R=3$, $R$ is $\Q$-Gorenstein, and $(R,\Delta)$ is non-weakly exceptional.
\end{enumerate}
One can apply Proposition~\ref{alphabcmf} to these cases.
\end{rem}

Next, we give some estimates of the local BCM-alpha invariant for complete regular local rings.
The first statement of the following proposition is a counterpart of \cite[Corollary~7.8]{HLS} for BCM-test ideals.

\begin{prop}[{cf.~\cite[Corollary~7.8]{HLS}}]\label{RLRpositivity}
Let $(R,\m,k)$ be a complete regular local ring of residue characteristic $p>0$,
and let $\Delta\geq 0$ be a $\Q$-divisor on $\Spec R$.
\begin{enumerate}
\item 
If $\ord_{\m}(\Delta)<1$,
then
$\alphabcm(R,\Delta)\geq 1-\ord_{\m}(\Delta)$.
In particular, $\alphabcm(R,0)=1$.
\item If $k$ is $F$-finite, $\lfloor\Delta\rfloor=0$, and $\Supp\Delta$ is simple normal crossing,
then 
\[ \alphabcm(R,\Delta)=1-\max\{ \text{coefficients of $\Delta$}\}. \]
\end{enumerate}
\end{prop}

\begin{proof}
For (1), it is enough to show that $(R,\Delta=t\cdot\divi(r))$ is $\BCM_{\Rphat}$-regular by \cite[Theorem~8.11]{BMP2}.

First, suppose that $k$ is infinite.
We follow similar arguments to those in \cite[Proposition~9.5.13]{Laz} and \cite[Theorem~7.7]{HLS}.
When $\dim R=0$, the statement holds, since $R$ is a field and $\taubcm(R,\Delta)\neq 0$.
Now suppose that we know the statement for $(\dim R-1)$-dimensional cases.
Since $k$ is infinite, there exists $h\in\m\setminus {\m}^2$ such that $\ord_{\m}(\Delta)=\ord_{\m/hR}(\Delta|_{R/hR})$ by \cite[Proposition~8.5.7(3)]{HS}.
By \cite[Proposition~2.10 and Theorem~3.1]{MST},
\[
\tau_{\widehat{(R/hR)^{+}}}(R/hR,\Delta|_{R/hR}) \subset \tau_{\Rphat}(R,\Delta)\cdot (R/hR).
\]
Thus, the statement for $R$ follows by the induction hypothesis.

Suppose that $k$ is not infinite.
Let $C_k$ be a Cohen ring of $k$.
By Cohen's strucuture theorem, we may assume that
$R=k[[x_1,\cdots,x_d]]$ or 
$R=C_k[[x_1,\cdots,x_d]]/(f)$, where $f\in\m$ satisfies $f=0$ or $f\in\m\setminus\m^2$.
We only prove the latter case since the former case can be proven by similar arguments.

Take a Cohen ring $C_{\overline{k}}$ and a non-canonical inclusion $C_k\hookrightarrow C_{\overline{k}}$ by \cite[Theorem~4.1 and Theorem~6.2]{Cohen}.
Since $C_k$ is a DVR, $C_k\hookrightarrow C_{\overline{k}}$ is a faithfully flat morphism and hence splits as a $C_k$-linear map.
Set $R^{\prime}:=C_{\overline{k}}[[x_1,\cdots,x_d]]/(f)$.
Let $\m^{\prime}$ denote a unique maximal ideal of $R^{\prime}$ and $\Delta^{\prime}$ the pullbuck of $\Delta$ to $R^{\prime}$.
By \cite[Lemma~2.28]{INS} and the definition of $R^{\prime}$,
the induced map $R\hookrightarrow R^{\prime}$ splits as an $R$-linear map
and
$\ord_{\m}(\Delta)=\ord_{\m^{\prime}}(\Delta^{\prime})$.
Consider the commutative diagram
\[
  \xymatrix{
    R    \ar[r] \ar[d] & \Rphat  \ar[d]  \\
    R^{\prime} \ar[r]   & \widehat{(R^{\prime})^{+}},
  }
\]
where the vertical arrows map $1$ to $r^t$.
Since the left vertical arrow splits, if the bottom horizontal arrow is pure, then the top horizontal arrow is so.
Since $(R^{\prime},\Delta^{\prime})$ is $\BCM_{\widehat{{{R}^{\prime}}^+}}$-regular,
$(R,\Delta)$ is $\BCM_{\Rphat}$-regular by \cite[Lemma~7.1]{MST}.

Next, we prove (2).
Set $d:=\dim R$.
Write $\Delta=\sum_{i=1}^{d} a_i\cdot \divi(x_i)$, where $x_1,\cdots,x_d$ is a regular parameter system.
It is enough to show that
\[ \alphabcm(R,\Delta)\geq 1-\max_{i}{a_i} \]
since the reverse inequality easily follows from the fact that BCM-regular pairs are klt \cite[Corollary~6.22]{MS}.
When $\dim R\leq 1$, this follows from (1).
Let $\pi:Y\rightarrow X=\Spec R$ be the blow-up of $\m$.
By applying Lemma~\ref{alphabcmf} in this situation,
\[
\alphabcm(R,\Delta) 
\geq \alpha_F (\PP^{d-1},\sum a_i \cdot \divi(T_i);\sO_{\PP^{d-1}}(1))
\]
since $\ord_\m$ is identical to the vanishing order of a unique prime $\pi$-exceptional divisor.
By \cite[Proposition~3.21]{TY} and Fedder's criterion \cite[Lemma~3.9(1)]{takagi},
\begin{equation*}
\alpha_F (\PP^{d-1},\sum a_i \cdot \divi(T_i);\sO_{\PP^{d-1}}(1))\geq 1-\max_{i}a_i . \qedhere
\end{equation*}
\end{proof}

\section{Comparison of non-nef loci and restricted base loci}

In this section, we prove the theorems stated in the Introduction (Theorem~\ref{maintheoremzero}, \ref{maintheoremp}, \ref{maintheoremtestideal}, \ref{maintheorem}, and Corollary~\ref{3dimkltcase}).
Note that we have already shown in Proposition~\ref{easycontainment} that the containment $\NNef(D/S)\subset\Bm(D/S)$ holds in general.

\subsection{Equal characteristic zero}

Let $S$ be an integral excellent $\Q$-scheme admitting a dualizing complex, and let $f:X\rightarrow S$ be a projective surjective morphism from a normal integral scheme.
In this section, we use (asymptotic) multiplier ideals defined in Section~\ref{secmult}.

\begin{prop}[{cf.~\cite[Lemma~4.1 and Corollay~4.4]{CdB}}]\label{compzero}
Suppose that $S$ is affine. Let $(X,\Delta)$ be a log $\Q$-Gorenstein pair, and let $D$ be a big $\Q$-Cartier $\Q$-divisor on $X$.
Then
\[ \Bm (D) \subset \bigcup_{m\in\Z_{>0}} \Zero(\J(X,\Delta,\|mD\|)), \]
\[ \bigcup_{m\in\Z_{>0}} \Zero(\J(X,\Delta,\|mD\|)) \setminus\Zero(\J(X,\Delta)) \subset \NNef(D). \]
\end{prop}

\begin{proof}
We follow the same arguments as those in \cite[Lemma~4.1-Corollay~4.4]{CdB}.

First, we prove the former containment.
Suppose $x\in X\setminus\bigcup_{m\in\Z_{>0}} \Zero(\J(X,\Delta,\|mD\|)).$
Take any globally generated and ample Cartier divisor $A$ on $X$ and any ample Cartier divisor $H$ on $X$ such that $H-(K_X +\Delta)-(\max_{s\in S}\dim X_{s})A$ is nef and big.
Then 
$\J(X,\Delta,\|mD\|)\otimes\sO_{X}(mD+H)$ is globally generated for every $m\in\Z_{>0}$ by Proposition~\ref{uggmult}.
Since ${\J(X,\Delta,\|mD\|)}_{x}=\sO_{X,x}$,
$mD+H$ is globally generated at $x$.
Thus, $x\not\in\Bm(D)$.

The latter containment follows from Proposition~\ref{regklt} and the same arguments as those in the proofs of \cite[Proposition~4.2-Corollay~4.4]{CdB}.
\end{proof}

\begin{thm}[{cf.~\cite[Corollay~4.9]{CdB}}]\label{maintheoremzero}
If $X$ has only klt type singularities except for finitely many closed points, then
\[ \NNef(D/S) = \Bm(D/S) \]
for every $\R$-Cartier $\R$-divisor $D$ on $X$.
\end{thm}

\begin{proof}
This follows from Proposition~\ref{easycontainment}, \ref{compzero}, and Corollary~\ref{clpt}(2).
\end{proof}

\subsection{Characteristic $p>0$}

Let $S$ be an integral and $F$-finite $\F_p$-scheme, and let $f:X\rightarrow S$ be a projective surjective morphism from a normal integral scheme.
In this section, we use (asymptotic) test ideals defined in Section~\ref{sectest}.

\begin{prop}[{cf.~\cite[Theorem~4.5]{Sato}}]
Suppose that $S$ is affine. Let $(X,\Delta)$ be a log $\Q$-Gorenstein pair, and let $D$ be a big $\Q$-Cartier $\Q$-divisor on $X$.
Then
\[ \Bm (D) \subset \bigcup_{m\in\Z_{>0}} \Zero(\tau(X,\Delta,\|mD\|)), \]
\[ \bigcup_{m\in\Z_{>0}} \Zero(\tau(X,\Delta,\|mD\|)) \setminus\Zero(\tau(X,\Delta)) \subset \NNef(D). \]
\end{prop}

\begin{proof}
The first containment follows from Proposition~\ref{uggtest} and the same arguments as those in Proposition~\ref{compzero}.
The second containment follows from \cite[Proposition~3.18]{Sato} or Proposition~\ref{stability}, and the same arguments as those in \cite[Proposition~4.5 (2) implies (3)]{Sato}. (See also the proof of Theorem~\ref{maintheoremtestideal}.)
\end{proof}

Similarly to the case of equal characteristic zero, the following theorem is obtained.

\begin{thm}[{cf.~\cite[Corollary~4.8]{Sato}}]\label{maintheoremp}
If $X$ has only strongly $F$-regular singularities except for finitely many closed points, then
\[ \NNef(D/S) = \Bm(D/S) \]
for every $\R$-Cartier $\R$-divisor $D$ on $X$.
\end{thm}

\subsection{Mixed characteristic}

In this section, we use (asymptotic) test ideals defined in Section~\ref{sectestmixed}.

\begin{prop}\label{uggmixedcontainment}
Let $(V,\m)$ be a DVR of mixed characteristic.
Let $X$ be a normal integral scheme projective and flat over $V$, and let $(X,\Delta)$ be a log $\Q$-Gorenstein pair. 
Then
\[ \Bm (D) \subset \bigcup_{m\in\Z_{>0}} \Zero(\tau(X,\Delta,\|mD\|)) \]
for every big $\Q$-Cartier $\Q$-divisor $D$ on $X$.
\end{prop}

\begin{proof}
This follows from Proposition~\ref{uggmixed} and the same arguments as those in Proposition~\ref{compzero}.
\end{proof}

\begin{thm}\label{maintheoremtestideal}
With notation as above,
\[
\NNef(D)\setminus\Zero(\tau(X,\Delta))
=\Bm(D)\setminus\Zero(\tau(X,\Delta))
=\bigcup_{m\in\Z_{>0}} \Zero(\tau(X,\Delta,\|mD\|))\setminus\Zero(\tau(X,\Delta)).
\]
\end{thm}

\begin{proof}
By Proposition~\ref{easycontainment} and \ref{uggmixedcontainment}, it suffices to show that
\[ \Zero(\tau(X,\Delta,\|mD\|))\setminus\Zero(\tau(X,\Delta)) \subset \NNef(D). \]
We may assume that $m=1$ since $\NNef(D)=\NNef(mD)$.
Suppose that there exists $x\in \Zero(\tau(X,\Delta,\|D\|))\setminus (\Zero(\tau(X,\Delta))\cup \NNef(D))$.
When $x$ is a point of the generic fiber of $f$,
\[
x\in \Zero(\J(X_{K(V)},\Delta_{K(V)},\|D_{K(V)}\|))\setminus (\Zero(\J(X_{K(V)},\Delta_{K(V)}))\cup \NNef(D_{K(V)}))
\]
by Proposition~\ref{uniftest}(2) and \ref{bcnnef}.
This contradicts Proposition~\ref{compzero}.
Thus, we may assume that $x$ is a point of the closed fiber of $f$.
Since $x\not\in\NNef(D)$, for any $\delta\in\Q_{>0}$, there exist $m\in\Z_{>0}$ and $G\in|mD|$ such that $\frac{1}{m}\ord_{\m_x}(G)<\delta$.
Since $x\not\in\Zero(\tau(X,\Delta,\|D\|))$,
${\tau(X,\Delta+\frac{1}{m}G)}_{x}\subset\tau(X,\Delta,\|D\|)_{x}\subsetneq \sO_{X,x}$.
By Proposition~\ref{uniftest}(3),
$(\widehat{\sO_{X,x}},\widehat{\Delta_x})$ is BCM-regular and
$(\widehat{\sO_{X,x}},\widehat{\Delta_x}+\frac{1}{m}\widehat{G_x})$ is not BCM-regular.
This contradicts Proposition~\ref{stability}.
\end{proof}

\begin{thm}\label{maintheorem}
Let $S$ be a Dedekind scheme of characteristic zero, and let $f:X \rightarrow S$ be a projective surjective morphism from a normal integral scheme.
Suppose that every non-closed point $x\in X$ satisfies one of the following conditions:
\begin{enumerate}
\item $(x\in X)$ is of equal characteristic zero and of klt type.
\item $(x\in X)$ is of mixed characteristic and of BCM-regular type.
\end{enumerate}
Then
\[ \NNef(D/S)=\Bm(D/S) \]
for every $\R$-Cartier $\R$-divisor $D$ on $X$.
\end{thm}

\begin{proof}
We may assume that $S$ is the spectrum of a DVR $(V,\m)$ of characteristic zero and $D$ is big $\Q$-Cartier by Proposition \ref{bcbaselocus}, \ref{bcnnef}, \ref{nneffact}, and Remark~\ref{baselocusfacts}(3).
By Corollary~\ref{clpt}(2), it suffices to show that for every non-closed point $x\in X$, if $x\in\Bm(D)$, then $x\in\NNef(D)$.
When $(x\in X)$ is of equal characteristic zero, this follows from Condition (1) and Corollary~\ref{maintheoremzero}.
Thus, we may assume that $(x\in X)$ is of mixed characteristic.
By Condition (2) and \cite[Lemma~5.5]{HLS}, there exists a $\Q$-divisor $\Delta\geq 0$ on $X$ such that $(X,\Delta)$ is log $\Q$-Gorenstein and $(\widehat{\sO_{X,x}},\widehat{\Delta_x})$ is BCM-regular.
Therefore, $x\in\NNef(D)$ follows from Proposition~\ref{uniftest}(3) and Theorem~\ref{maintheoremtestideal}.
\end{proof}

\begin{cor}\label{3dimkltcase}
Let $S$ be a Dedekind scheme of characteristic zero whose residue fields of positive characteristic are $F$-finite and of characteristic $p>5$, and let $f:X \rightarrow S$ be a projective surjective morphism from a normal integral scheme. Suppose that $\dim X=3$ and $X$ has only klt type singularities except for finitely many closed points.
Then
\[ \NNef(D/S)=\Bm(D/S) \]
for every $\R$-Cartier $\R$-divisor $D$ on $X$.
\end{cor}

\begin{proof}
Let $x\in X$ be a non-closed point.
When $(x\in X)$ is of mixed characteristic,
by \cite[Theorem~7.11]{MST},
$\sO_{X,x}$ is BCM-regular
since $\sO_{X,x}$ is of residue characteristic $p>5$, at most $2$-dimensional, and klt. Thus, the statement follows from Theorem~\ref{maintheorem}.
\end{proof}

\bibliographystyle{alpha}
\bibliography{non_nef_loci}

\end{document}